\documentclass[a4paper,12pt]{amsart}
\usepackage[latin1]{inputenc}
\usepackage{amsbsy}
\usepackage{amssymb}
\usepackage{mathrsfs}
\usepackage{eucal}

\DeclareMathOperator{\gauss}{\mathcal{G}}
\DeclareMathOperator{\Spec}{Spec}
\DeclareMathOperator{\Proj}{Proj}
\DeclareMathOperator{\Hess}{Hess}
\DeclareMathOperator{\Sing}{Sing}
\DeclareMathOperator{\Reg}{Reg}
\DeclareMathOperator{\Rad}{Rad}
\DeclareMathOperator{\kar}{char}
\DeclareMathOperator{\Diff}{Diff}

\title[Integral points on hypersurfaces of degree at least four]{The density of integral points on hypersurfaces of degree at least four}
\author{Oscar Marmon}
\date{}
\address{Mathematical Sciences \\ Chalmers University of Technology \\ SE-412 96 Gothenburg}
\address{Mathematical Sciences \\ University of Gothenburg \\ SE-412 96 Gothenburg}
\email{marmon@chalmers.se}

\begin{document}


\newcommand{\xx}{\mathbf{x}}
\newcommand{\xxi}{\boldsymbol{\xi}}
\newcommand{\eeta}{\boldsymbol{\eta}}
\newcommand{\0}{\boldsymbol{0}}
\renewcommand{\aa}{\mathbf{a}}
\newcommand{\bb}{\mathbf{b}}
\newcommand{\uu}{\mathbf{u}}
\newcommand{\vv}{\mathbf{v}}
\newcommand{\yy}{\mathbf{y}}
\newcommand{\zz}{\mathbf{z}}
\renewcommand{\SS}{\mathbf{S}}
\newcommand{\TT}{\mathbf{T}}
\newcommand{\ZZ}{\mathbb{Z}}
\newcommand{\Zpol}{\ZZ[X_1,\ldots,X_n]}
\newcommand{\FF}{\mathbb{F}}
\newcommand{\RR}{\mathbb{R}}
\renewcommand{\AA}{\mathbb{A}}
\newcommand{\CC}{\mathbb{C}}
\newcommand{\PP}{\mathbb{P}}
\newcommand{\GG}{\mathbb{G}}
\newcommand{\QQ}{\mathbb{Q}}
\newcommand{\sB}{\mathsf{B}}
\newcommand{\cH}{\mathcal{H}}
\newcommand{\cP}{\mathcal{P}}
\newcommand{\cM}{\mathcal{M}}
\newcommand{\cN}{\mathcal{N}}
\newcommand{\cI}{\mathcal{I}}
\newcommand{\cU}{\mathcal{U}}
\newcommand{\cR}{\mathcal{R}}
\newcommand{\cQ}{\mathcal{Q}}
\newcommand{\cS}{\mathcal{S}}
\newcommand{\cB}{\mathcal{B}}
\newcommand{\cE}{\mathcal{E}}
\newcommand{\R}{\mathrm{R}}
\newcommand{\epsi}{\varepsilon}
\newcommand{\smod}[1]{\,(#1)}
\newcommand{\spmod}[1]{\,(\bmod{#1})}
\renewcommand{\theenumi}{\roman{enumi}}
\newtheorem{lemma}{Lemma}[section]
\newtheorem{prop}{Proposition}[section]
\newtheorem{thm}{Theorem}[section]
\newtheorem*{thm*}{Theorem}
\newtheorem{claim}{Claim}[section]
\newtheorem{cor}{Corollary}[section]
\theoremstyle{remark}
\newtheorem*{notation*}{Notation}
\newtheorem*{note*}{Note}
\newtheorem{note}{Note}
\newtheorem*{rem*}{Remark}
\newtheorem{rem}{Remark}[section]
\theoremstyle{definition}
\newtheorem*{def*}{Definition}
\newtheorem{definition}{Definition}[section]
\newtheorem{notation}{Notation}[section]


\begin{abstract}
Let $f$ be a polynomial of degree at least four with integer-valued coefficients. We establish new bounds for the density of integer solutions to the equation $f=0$, using an iterated version of Heath-Browns $q$-analogue of van der Corput's method of exponential sums.  
\end{abstract}


\maketitle

\section{Introduction}

Given a polynomial $f \in \ZZ[x_1,\dotsc,x_n]$ we wish to study the solutions in $\ZZ^n$ to the Diophantine equation
\begin{equation}\label{eq:f=0}
 f(x_1,\dotsc,x_n)=0.
\end{equation}
We are interested in the density of solutions, that is, for a given positive real number $B$ we want to estimate the number of solutions $\xx$ to \eqref{eq:f=0} satisfying $\vert \xx \vert \leq B$, where $\vert \xx \vert = \max_i \vert x_i \vert$. To this end we introduce the \emph{counting function}
\[
 N(f,B) = \#\{ \xx \in \ZZ^n ; f(\xx)=0, \vert \xx \vert \leq B \}.
\]
We shall use congruences as a tool to estimate $N(f,B)$. Thus, we introduce the counting functions
\[
 N(f,B,m) = \#\{ \xx \in \ZZ^n ; f(\xx) \equiv 0 \spmod m,\vert \xx \vert \leq B\}.
\]
Trivially, for any $m \in \ZZ$, $N(f,B,m)$ is an upper bound for $N(f,B)$. 
We extend this notation to systems of equations in the obvious way:
\begin{gather*}
 N(f_1,\dotsc,f_r,B) = \#\{ \xx \in \ZZ^n ; f_1(\xx)=\dotsb=f_r(\xx)=0, \vert \xx \vert \leq B\},\\
 \begin{split}
 N(f_1,\dotsc,f_r,B,m) = \#\{ &\xx \in \ZZ^n ; \\
 & f_1(\xx) \equiv \dotsb \equiv f_r(\xx) \equiv 0 \spmod m, \vert \xx \vert \leq B \}.
 \end{split}
\end{gather*}
By the \emph{leading form} of the polynomial $f$ we shall mean the homogeneous part of maximal degree.
Heath-Brown \cite{Heath-Brown} proved that for a polynomial $f\in\Zpol$ of degree at least 3 such that the leading form $F$ is non-singular (i.e. defines a non-singular hypersurface in $\PP^n_\CC$), we have the estimate
\[
 N(f,B) \ll_F B^{n-3+15/(n+5)}
\]
for $n\geq5$. To prove this, Heath-Brown studied $N(f,B,pq)$ for two different primes $p,q$, and devised a version of van der Corput's method of exponential sums as a key step in the estimation of this counting function. By incorporating an exponential sum estimate by Katz \cite{Katz} into Heath-Brown's method, the author \cite{Marmon} sharpened this result slightly, to
\[
 N(f,B) \ll_F B^{n-3+(13n-8)/(n^2+3n-2)}(\log B)^{n/2}
\]
for $n\geq6$. Salberger \cite{Salberger-Integral_points} was able to sharpen the estimate further, through a new geometric argument. He proved
\begin{equation}
\label{eq:salberger}
 N(f,B) \ll_F B^{n-3+9/(n+2)} (\log B)^{n/2}
\end{equation}
for $n\geq4$. 

For polynomials of degree at least 4, one can try to iterate the Weyl (or van der Corput) differencing step in \cite{Heath-Brown} twice to get even sharper estimates, and that is the approach we will take in this paper. The aim is to prove the following result:

\begin{thm}
\label{thm:main}
Let $f$ be a polynomial in $\ZZ[x_1,\dotsc,x_n]$ of degree $d\geq4$ with leading form $F$. Let  $Z=\Proj\ZZ[x_1,\dotsc,x_n]/(F)$, and suppose that $Z_\QQ$ is a non-singular subscheme of $\PP^{n-1}_\QQ$. Then we have the estimate
\[
N(f,B)\ll_{F} B^{n-4+(37n-18)/(n^2+8n-4)}.
\]
\end{thm}

The estimate in Theorem \ref{thm:main} improves upon \eqref{eq:salberger} as soon $n \geq 17$. Moreover, if $n \geq 29$, Theorem \ref{thm:main} implies that $N(f,B) \ll_{F} B^{n-3}$.

Using an argument of Heath-Brown, we can derive a uniform version of Theorem \ref{thm:main}.

\begin{thm}
\label{thm:mainuniform}
Let $f$ be a polynomial in $\ZZ[x_1,\dotsc,x_n]$ of degree $d\geq4$ with leading form $F$. Let  $Z=\Proj\ZZ[x_1,\dotsc,x_n]/(F)$, and suppose that $Z_\QQ$ is a non-singular subscheme of $\PP^{n-1}_\QQ$. Then we have the estimate
\[
N(f,B)\ll_{n,d,\epsi} B^{n-4+(37n-18)/(n^2+8n-4)} + B^{n-3+\epsi}
\]
for any $\epsi > 0$.
\end{thm}

When proving these two theorems, it will be convenient to seek to estimate a weighted counting function rather than the original one. More precisely, let $W:\RR^n\to [0,1]$ be an infinitely differentiable function, supported on $[-2,2]^n$. 
Then we define weighted counting functions
\[
 N_W(f,B,m) = \sum_{\substack{\xx\in\ZZ^n\\m \mid f(\xx)}}W\left(\frac{1}{B}\xx\right).
\]
In the proof of Theorem \ref{thm:main} we shall take $W$ to be the function defined by
\begin{equation}
\label{eq:W}
 W(\mathbf t) = \prod_{i=1}^n w(t_i/2), \text{ where } 
 w(t)=
 \begin{cases}
  \exp(-1/(1-t^2)), & \vert t \vert < 1,\\
  0, & \vert t \vert \geq 1.
 \end{cases}
\end{equation}
It is then clear that $N(f,B,m) \ll N_W(f,B,m)$. Approximating the characteristic function of the cube $[-B,B]^n$ with a smooth function in this way allows us to sharpen some of the estimates involved.

The proof of Theorem \ref{thm:main} is carried out in Sections \ref{sec:technicalheart} and \ref{sec:proof}, and incorporates the idea of Salberger (see Remarks \ref{rem:salberger's idea} and \ref{rem:second diff}). We shall use a modulus which is a product of three distinct primes $m=\pi pq$, where the primes $\pi,p$ can be viewed as parameters connected to the two consecutive differencing steps. The two differencings put us in the position to apply results on the density of $\FF_q$-rational points on a family of new varieties over $\FF_q$, parameterized by integral $n$-tuples $\yy,\zz$. These results, behind which lie Deligne's bounds for exponential sums over non-singular varieties, become weaker as the dimensions of the singular loci of the varieties increase, and thus we need to control these dimensions. Section \ref{sec:geometry} is devoted to this problem. 


\section{Preliminary geometric results}
\label{sec:geometry}

The geometric arguments in this section extend those of Salberger \cite{Salberger-Integral_points}. A priori, some of our results are valid in characteristic zero only, but in \ref{subsec:charp} we obtain conditions on primes $p$ ensuring the truth of the statements in characteristic $p$.

\subsection{Results for polynomials over a field}
\label{subsec:char0}
In this section, suppose that $K$ is a field. Let $\kar K = p$. Furthermore, we shall assume that $n \geq 3$.

\begin{notation}
If $F \in K[x_1,\dotsc,x_n]$ is a homogeneous polynomial and $\yy \in K^n$, we define
\[
F^\yy(\xx)=\yy\cdot\nabla F(\xx)=y_1\frac{\partial F}{\partial x_1}+\dotsb+y_n\frac{\partial F}{\partial x_n}.
\]
Furthermore, for each pair $\yy,\zz$ of $n$-tuples of elements of $K$, we define
\[
F^{\yy,\zz}(\xx)=(\Hess (F))\yy\cdot \zz=\sum_{1\leq i,j\leq n}\frac{\partial^2 F}{\partial x_i\partial x_j} y_i z_j. 
\]
We have $F^{\yy,\zz} = (F^\yy)^\zz = (F^\zz)^\yy$.

For a collection $F_1,\dotsc,F_r$ of homogeneous polynomials we denote by $V(F_1,\dotsc,F_r)$ the closed subscheme of $\PP^{n-1}_K$ that they define. If $F,G$ are two homogeneous polynomials and $\zz \in K^n$, we define
\[
\Diff_\zz(F,G) = V(F,G,G^\zz).
\]
The reason for the notation is that the differencing process used in Section \ref{sec:proof} will lead us to consider such varieties. Note that the definition is not symmetric in $F$ and $G$.

When $x$ (or any other letter) is used to denote a $K$-point of $\PP^{n-1}_K$, we will use the corresponding bold letter $\xx$ to denote an element of $K^n$ representing $x$. Vice versa, given $\xx\in K^n \setminus \{\0\}$, we denote its homothety class by $x$.

We denote by $\GG(k,n-1)$ the set of $k$-dimensional linear subspaces of $\PP^{n-1}_K$.

Finally, we adopt the convention that the dimension of the empty variety is $-1$.
\end{notation}

\begin{definition}
If $V\subset\PP^{n-1}_K$ is a non-singular hypersurface defined by a homogeneous polynomial $G(x_1,\dotsc,x_n)$ of degree $d\geq 2$, then the \emph{Gauss morphism} $\gauss:V\to\PP^{n-1}_K$ is defined by $x\mapsto[\nabla G(\xx)]$, where $\nabla G(\xx)=(\frac{\partial G}{\partial x_1},\dotsc,\frac{\partial G}{\partial x_n})$. If $d$ is not divisible by $p$, it can be extended to the whole of $\PP^{n-1}_K$, since if $\nabla G(\xx)=\boldsymbol 0$ then $dG(\xx)=\xx \cdot \nabla G(\xx)=0$, so $G(\xx)=0$. Thus $\gauss$ is well-defined outside $V$. 
\end{definition}

\begin{rem}
\label{rem:finite}
It is easy to prove that the fibres of $\gauss$ are finite. In particular, this implies that the polynomial $G^\yy$, as defined above, cannot vanish identically for $\yy\neq\boldsymbol 0$, since then the image of $\PP^{n-1}_K$ under the Gauss map would be contained in a hyperplane.
\end{rem}

\begin{lemma}\label{lem:katz}
Let $X \subseteq \PP^{n-1}_K$ be an equidimensional subvariety of dimension $m$. Let $H \subset \PP^{n-1}_K$ be a hypersurface such that $X \cap H$ is equidimensional of dimension $m-1$. Then we have 
\[
(\Sing X) \cap H \subseteq \Sing(X \cap H).
\]
In particular,
\[
\dim \Sing (X \cap H) \geq \dim \Sing X -1.
\] 
\end{lemma}
This is a standard result, and we omit the proof.

\begin{notation}
Let $F,G \in K[x_1,\dotsc,x_n]$ be homogeneous polynomials, with $\deg G \geq 2$. For each $s=-1,0,\dotsc,n-1$, define $T_s(F,G)$ to be the Zariski closed subset of $z=[\zz]\in \PP^{n-1}_K$ such that 
\[
\dim (\Sing (\Diff_\zz(F,G))) \geq s.
\]

We define $T_{\deg}(F,G)$ to be the closed subset of $z \in \PP^{n-1}_K$ such that $\dim \Diff_\zz(F,G) = \dim V(F,G)$.
\end{notation}

We are interested in upper bounds for the dimension of $T_s(F,G)$. The version of Bertini's theorem that we shall use holds only in characteristic zero, whence the hypothesis in Lemma \ref{lem:T_s-nonsing}.

\begin{lemma}\label{lem:T_s-nonsing}
Suppose that $p = \kar K = 0$. Let $F,G \in K[x_1,\dotsc,x_n]$ be homogeneous polynomials, with $\deg G \geq 2$. Suppose that $Y=V(F,G)$ is a non-singular complete intersection of dimension $n-3$. Suppose furthermore that $V(G)$ is non-singular. Then, for $-1 \leq s \leq n-3$, we have
\[
\dim T_s(F,G) \leq n-2-s.
\]
(If $s>n-3$, then of course $T_s(F,G) = \varnothing$.)
\end{lemma}

\begin{proof}
Since $V(G)$ is non-singular, we can define the Gauss morphism 
 \[
 \gauss:\PP^{n-1}_K \to \PP^{n-1}_K,\xx \mapsto (\xi_1,\dotsc,\xi_n)=\left( \frac{\partial G}{\partial x_1},\dotsc,\frac{\partial G}{\partial x_n}\right). 
 \]
Note that, using the notation $H_\zz$ for the hyperplane $\zz \cdot \boldsymbol \xi=0$, we have $\Diff_\zz(F,G)=Y\cap\gauss^{-1}(H_\zz)$. We shall recursively find a sequence of linear subspaces $\Pi_{-1},\Pi_0,\dotsc,\Pi_{n-3}$ of $\PP^{n-1}_K$ such that $Y\cap\gauss^{-1}(\Pi_s)$ is non-singular of dimension $n-4-s$ for $s=-1,0,\dotsc,n-3$. Let $\Pi_{-1}=\PP^{n-1}_K$. Then $Y\cap\gauss^{-1}(\Pi_{-1})=Y$ is non-singular by assumption. Suppose next that we have already found a linear subspace $\Pi_s$, $s\in \{-1,0,\dotsc,n-4\}$ such that $Y_s:=Y\cap\gauss^{-1}(\Pi_s)$ is non-singular of dimension $n-4-s$, and let $\gauss_s:Y_s\to\Pi_s$ be the restriction of $\gauss$ to $Y_s$. Then, by Bertini's theorem \cite[Cor 6.11(2)]{Jouanolou}, we may find a hyperplane $\Pi_{s+1}\subset\Pi_s$ such that $\gauss_s^{-1}(\Pi_{s+1})=Y\cap\gauss^{-1}(\Pi_{s+1})$ is non-singular of dimension $n-5-s$. Here we use the fact that $K$ has characteristic zero.
 
 Now, for each $s=-1,0,\dotsc,n-3$, let $\Lambda_s$ be the $s$-dimensional linear subspace of $\PP^{n-1}_K=\Proj K[z_1,\dotsc,z_n]$ parameterizing hyperplanes $H_\zz$ such that $H_\zz\supseteq\Pi_s$. 
 We shall now prove that $T_s(F,G)\cap\Lambda_s=\varnothing$, and the statement will then follow from the projective dimension theorem. Therefore, suppose that $z=[\zz]\in\Lambda_s$. Since then $H_\zz\supseteq\Pi_s$, there is a linear subvariety $\Gamma_z\subseteq\PP^{n-1}_K$ of codimension $s$ such that $\Pi_s=H_\zz\cap \Gamma_z$. By the above, however,
\[
 Y\cap\gauss^{-1}(H_\zz)\cap\gauss^{-1}(\Gamma_z)=Y\cap\gauss^{-1}(\Pi_s)
\]
is non-singular, so by Lemma \ref{lem:katz} we must have
\begin{equation}
\label{eq:intersection}
 \left(\Sing (\Diff_\zz(F,G))\right) \cap \gauss^{-1}(\Gamma_z) = \varnothing.
\end{equation}
By Remark \ref{rem:finite} it follows that 
\[
\dim \gauss^{-1}(\Gamma_z)=\dim \Gamma_z=n-1-s.
\]
Now (\ref{eq:intersection}), along with the projective dimension theorem, implies that
\[
\dim(\Sing (\Diff_\zz(F,G)))\leq s-1.
\]
 Thus we have $z\not\in T_s(F,G)$, as promised.
\end{proof}

For the dimension of $T_{\deg}(F,G)$, we have the following result.

\begin{lemma}
\label{lem:T_deg}
Let $F,G \in K[x_1,\dotsc,x_n]$ be homogeneous polynomials, with $p \nmid \deg G \geq 2$. 
\begin{itemize}
\item[(i)]
Suppose that $Y=V(F,G)$ is a complete intersection of dimension $n-3$. Then we have
\[
\dim T_{\deg}(F,G) \leq 1.
\] 
\item[(ii)]
Suppose furthermore that $n \geq 4$, and that both $Y$ and $V(G)$ are non-singular. Then we have $T_{\deg}(F,G) = \varnothing$.
\end{itemize}
\end{lemma}
\begin{proof}
(i) 
As in the proof of Lemma \ref{lem:T_s-nonsing}, we have 
\[
\Diff_\zz(F,G) = Y \cap \gauss^{-1}(H_\zz).
\]
Thus, $z \in T_{\deg}(F,G)$ if and only if $W \subseteq \gauss^{-1}(H_\zz)$ for some irreducible component $W$ of $Y$. This means that $\gauss(W) \subseteq H_\zz$ for every $z \in T_{\deg}(F,G)$. By Remark \ref{rem:finite} we have $\dim \gauss(W) = \dim W = n-3$, so there is a linear subspace $L \subset \PP^{n-1}_K$ of dimension at least $n-3$ such that $L \subseteq H_\zz$ for all $z \in T_{\deg}(F,G)$. In other words,
\[
T_{\deg}(F,G) \subseteq \Sigma(L) := \{H \in \GG(n-2,n-1); L \subseteq H\}.
\]
We conclude that $\dim T_{\deg}(F,G) \leq \dim \Sigma(L) \leq 1$, proving (i).

(ii)
Since $V(G)$ is non-singular, $G^\zz$ does not vanish identically for $\zz\neq\boldsymbol 0$ by Remark \ref{rem:finite}. Thus it has degree $\deg(G)-1$. Moreover, since $Y$ is a non-singular complete intersection of dimension at least 1, it is geometrically integral. Let $Y_\zz = \Diff_\zz(F,G)$. 
 
Suppose now that $\dim Y_\zz = \dim Y$. If $\bar K$ is an algebraic closure of $K$, then we would also have $\dim (Y_\zz)_{\bar K} = \dim Y_{\bar K}$. Since $Y_{\bar K}$ is irreducible, this means that $V(G^\zz) \subseteq Y$, implying, by the homogeneous Nullstellensatz, that $G^\zz \in \Rad(F,G)$. However, the ideal $(F,G) \subset \bar K[x_1,\dotsc,x_n]$ is prime, hence radical, so we would have $G^\zz \in (F,G)$, which is impossible for degree reasons. This proves that $T_{\deg}(F,G)=\varnothing$.
\end{proof}

We shall now extend Lemma \ref{lem:T_s-nonsing} to the case of singular varieties. To this end, we shall use Bertini's theorem, in the following form.

\begin{lemma}\label{lem:bertini}
Suppose that $K$ is infinite. Let $X \subset \PP^{n-1}_K$ be a complete intersection of degree $d$ and dimension $m$. Put $\sigma = \dim \Sing X$. Then there exists a linear subspace $L \subseteq \PP^{n-1}_K$ of codimension $\sigma+1$, such that $X\cap L$ is non-singular, of degree $d$ and dimension $m-\sigma-1$.
\end{lemma}
\begin{proof}
By Bertini's theorem \cite[Cor. 6.11]{Jouanolou}, there exists a hyperplane $\Gamma \subset \PP^{n-1}_K$ such that
\begin{itemize}
\item[(i)] 
$\Gamma$ intersects each irreducible component of $X$ properly,
\item[(ii)]
$\Gamma$ intersects each irreducible component of $\Sing X$ properly,
\item[(iii)]
$(\Reg X) \cap \Gamma$ is non-singular. 
\end{itemize}
Repeating this process, we get the desired result. 
\end{proof}

In fact, one can show that '$K$ is infinite' may be replaced by '$K$ has cardinality greater than some constant depending only on $n$ and $d$'. In the finite field case, one could then use the effective Bertini theorem proved by Ballico \cite{ballico}.

\begin{lemma}
\label{lem:T_s-arbitrary}
Let $F,G \in K[x_1,\dotsc,x_n]$ be homogeneous polynomials with $\deg G \geq 2$. Suppose that $Y=V(F,G)$ is a complete intersection of dimension $n-3$. Let $\tilde Y = V(G)$ and define
\[
\sigma = \max \{\dim \Sing Y, \dim \Sing \tilde Y\}.
\]
\begin{itemize}
\item[(i)]
Suppose that $p=0$. 
Then, for $-1 \leq s \leq n-3$ we have
\[
\dim T_{\sigma+s+1}(F,G) \leq n-2-s.
\]
\item[(ii)]
Suppose that $n \geq 5$ and $p \nmid \deg G$. Then we have
\[
\dim T_{\deg}(F,G) \leq \min\{\sigma,1\}.
\]
\end{itemize}
\end{lemma}
\begin{proof}
(i) In case $\sigma = -1$, the statement follows directly from Lemma \ref{lem:T_s-nonsing}, so we assume that $\sigma \geq 0$. By Lemma \ref{lem:bertini} we can find a linear subspace $L \subset \PP^{n-1}_K$ of codimension $\sigma +1$ such that $Y \cap L$ and $\tilde Y \cap L$ are non-singular. $L$ can be chosen in such a way that the degrees of the varieties are preserved and the dimensions decrease by $\sigma +1$.

Without loss of generality, assume that $L$ is given by $x_n=x_{n-1}=\dotsb=x_{n-\sigma}=0$. Then $Y_0=Y\cap L$ and $\tilde Y_0=\tilde Y \cap L$ are non-singular subschemes of $\PP^{n-\sigma-2}_K=\Proj K[x_1,\dotsc,x_{n-\sigma-1}]$. We have $Y_0 = V(F_0,G_0)$ and $\tilde Y_0 = V(G_0)$, where  
\begin{gather*}
F_0(x_1,\dotsc,x_{n-\sigma-1}) = F(x_1,\dotsc,x_{n-\sigma-1},0,\dotsc,0), \\
G_0(x_1,\dotsc,x_{n-\sigma-1}) = G(x_1,\dotsc,x_{n-\sigma-1},0,\dotsc,0).
\end{gather*}
For every $\zz=(z_1,\dotsc,z_{n-\sigma-1},0,\dotsc,0)\in L$, we have $\Diff_\zz(F,G)\cap L=\Diff_{\zz_0}(F_0,G_0)$, where $\zz_0=(z_1,\dotsc,z_{n-\sigma-1})$. 

By repeated application of Lemma \ref{lem:katz}, we have
\[
T_{\sigma+s+1}(F,G) \cap L \subseteq T_s(F_0,G_0),
\]
and by Lemma \ref{lem:T_s-nonsing} we have
\[
\dim T_{s}(F_0,G_0)\leq n-(\sigma+1)-2-s.
\] 
Hence $\dim T_{\sigma+s+1}(F,G)\leq n-2-s$ by the projective dimension theorem.

(ii) In case $\sigma \neq 0$, the statement follows directly from Lemma \ref{lem:T_deg}. Thus, suppose that $\sigma = 0$. Since $\dim T_{\deg}(F,G) = \dim T_{\deg}(F,G) \otimes {\bar K}$, we may assume that $K$ is infinite, and apply the construction above with a hyperplane $L \subset \PP^{n-1}_K$. One easily sees that 
\[
T_{\deg}(F,G) \cap L \subseteq T_{\deg}(F_0,G_0) = \varnothing,
\]
which implies that $\dim T_{\deg}(F,G) \leq 0$.
\end{proof}

We shall apply the results above in the case when $G=F^\yy$. Thus, we introduce the following notation. 

\begin{notation}
\label{not:V_y}
If $V=V(F)$ is a hypersurface of degree at least 3 in $\PP^{n-1}_K$, define
\begin{gather*}
V_\yy = V(F,F^\yy), \quad \tilde V_\yy = V(F^\yy),
\end{gather*}
for any $\yy \in K^n$, and let 
\begin{gather*}
s_\yy(V) = \dim \Sing V_\yy,\quad \tilde s_\yy(V) = \dim \Sing \tilde V_\yy,\\
\sigma_\yy(V) = \max \{s_\yy(V), \tilde s_\yy(V)\}.
\end{gather*}
Define $T_\sigma(V)$, for any $-1 \leq \sigma\leq n-1$, as the closed subset of $y \in \PP^{n-1}_K$ such that $\sigma_\yy(V) \geq \sigma$.

For any pair $(\yy,\zz) \in K^n \times K^n$, we define
\[ 
V_{\yy,\zz} = \Diff_\zz(F,F^\yy) = V(F,F^\yy,F^{\yy,\zz}).
\]
Furthermore, let $T_{\deg,\yy}(V) = T_{\deg}(F,F^\yy)$, and $T_{s,\yy}(V) = T_s(F,F^\yy)$ for $-1 \leq s \leq n-1$.
\end{notation}

\begin{lemma}
\label{lem:dimV_y}
Let $V$ be a non-singular hypersurface of degree $d \geq 3$ in $\PP^{n-1}_K$. Then we have
\[
\dim V_\yy = n-3
\]
for any $\yy \in K^n \setminus \{\mathbf 0\}$.
\end{lemma}
\begin{proof}
Let $F \in K[x_1,\dotsc,x_n]$ be a generator for the ideal of $V$. Since $V$ is non-singular, $F^\yy$ does not vanish identically by Remark \ref{rem:finite}, and thus has degree $d-1$. Moreover, since $V$ is non-singular of dimension at least $1$, it is geometrically integral.

Suppose now that $\dim V_\yy = n-2$. If $\bar K$ is an algebraic closure of $K$, then we would also have $\dim (V_\yy)_{\bar K} = n-2$. Since $V_{\bar K}$ is irreducible, this means that $V_{\bar K} \subseteq (V_\yy)_{\bar K}$, implying, by the homogeneous Nullstellensatz, that $F^\yy \in \Rad_{\bar K}(F) = (F)$. This is impossible for degree reasons. Thus $\dim V_\yy = n-3$.
\end{proof}

Applying Lemma \ref{lem:T_deg} and Lemma \ref{lem:T_s-arbitrary} in this case we get the following result.  

\begin{lemma}
\label{lem:T_s-special}
Let $V$ be a non-singular hypersurface of degree $d \geq 3$ in $\PP^{n-1}_K$. Let $\yy \in K^n \setminus \{\mathbf 0\}$ and put $\sigma_\yy = \sigma_\yy(V)$. 
\begin{itemize}
\item[(i)]
Suppose that $p=0$. Then, for $-1 \leq s \leq n-1$ we have
\[
\dim T_{\sigma_\yy+s+1,\yy}(V) \leq n-2-s.
\]
\item[(ii)]
Suppose that $n \geq 5$ and $p \nmid (d-1)$. Then we have
\[
\dim T_{\deg,\yy}(V) \leq \min\{\sigma_\yy,1\}.
\] 
\end{itemize}
\end{lemma}

\begin{proof}
By Lemma \ref{lem:dimV_y} we have $\dim V_\yy = n-3$. Now (i) is trivially true for $s \geq n-2$, and otherwise follows from Lemma \ref{lem:T_s-arbitrary}. Part (ii) follows from Lemma \ref{lem:T_deg}.
\end{proof}


\subsection{Results for polynomials over $\ZZ$}
\label{subsec:charp}

We have proved part (i) of Lemma \ref{lem:T_s-special} only in characteristic 0. The aim of this section is to show that it is also true in characteristic $p$ for large enough $p$. Assume throughout this section that $n \geq 3$.

\begin{notation}
\label{not:fibers}
If $f:X \to \Spec \ZZ$ is a morphism of schemes, we denote $f^{-1}((0))$ by $X_0$, and $f^{-1}((p))$ by $X_p$ for a prime $p \in \ZZ$.
\end{notation}

\begin{definition}
\label{def:R_i(p)}
Suppose that $F \in \ZZ[x_1,\dotsc,x_n]$ is a homogeneous polynomial of degree $d \geq 3$. Let $p$ be a prime number or $p=0$. Let $Z=\Proj\ZZ[x_1,\ldots,x_n]/(F)$. Recall Notation \ref{not:V_y}. We say that $F$ (or $Z$) satisfies the property  
\begin{itemize}
\item[$(\mathrm R_0(p))$] 
if $Z_{p}$ is a non-singular variety.
\item[$(\mathrm R_1(p))$]
if for every $s=-1,0,\dotsc,n-1$,
\[
  \dim T_s(Z_p)\leq n-2-s.
\]
\item[$(\mathrm R_2(p))$]
if for every $\yy\in\FF_p^n$ and every $s=-1,0,\dotsc,n-1$, 
\[
  \dim T_{\sigma_\yy(Z_p)+s+1,\yy}(Z_{p})\leq n-2-s.
\]  
\end{itemize}
\end{definition}

In section \ref{subsec:char0} it was shown that $(\mathrm R_0(0))$ implies $(\mathrm R_2(0))$. Combining the geometric results in \cite{Salberger-Integral_points} with \cite[Lemma 2]{Heath-Brown} one sees that if $F$ satisfies $(\mathrm R_0(0))$, then $F$ satisfies $(\mathrm R_0(p))$ and $(\mathrm R_1(p))$ as soon as $p$ is large enough. Our aim in this section is to show the corresponding result for $(\mathrm R_2(p))$.

\begin{notation}
\label{not:hilbert scheme}
Let $\cH$ be the Hilbert scheme parameterizing degree $d$ hypersurfaces in $\PP^{n-1}_\ZZ$. $\cH$ can be identified with $\PP^{D}_\ZZ$, where $D = \binom{n-1+d}{d}-1$, and homogeneous coordinates for $\cH$ are given by $\mathbf t=(t_I)$, where $I$ runs over all $n$-tuples $(i_1,\ldots,i_n)$ of non-negative integers such that $i_1+\dotsb+i_n=d$. If $\xx=(x_1,\dotsc,x_n)$ are homogeneous coordinates for $\PP^{n-1}_\ZZ$, then $\xx^I$ denotes the monomial $x_1^{i_1}\dotsm x_n^{i_n}$. 

Furthermore, let
\[
\mathcal P=\cH\times\PP^{n-1}_\ZZ\times\PP^{n-1}_\ZZ\times\PP^{n-1}_\ZZ.
\]
\end{notation}

\begin{notation}
\label{not:universal hypersurface}
Introduce multihomogeneous coordinates $(\aa,\yy,\zz,\xx)$ on $\cP$. Consider the following multihomogeneous polynomials: 
\begin{gather*}
 F(\mathbf a,\yy,\zz,\xx)=\sum a_I \xx^I,\\
 G(\mathbf a,\yy,\zz,\xx)= \sum y_i \frac{\partial F}{\partial x_i},\\
 H(\mathbf a,\yy,\zz,\xx)= \sum z_j \frac{\partial G}{\partial x_j} = \sum_{i,j} y_i  \frac{\partial^2 F}{\partial x_i\partial x_j} z_j.
\end{gather*}
\begin{itemize}
\item[(i)]
Let $\mathcal M$ be the closed subscheme of $\mathcal P$ defined by $F,G,H$ and all $3\times3$-minors of the matrix
\[
 \begin{bmatrix}
  {\partial F}/{\partial x_1} & \cdots & {\partial F}/{\partial x_{n}}\\
  {\partial G}/{\partial x_1} & \cdots & {\partial G}/{\partial x_{n}}\\
  {\partial H}/{\partial x_1} & \cdots & {\partial H}/{\partial x_{n}}
 \end{bmatrix},
\]
and let $\pi_{\cM}:\mathcal M\to \cH\times\PP^{n-1}_\ZZ\times\PP^{n-1}_\ZZ$ be the projection onto the first three factors.
\item[(ii)]
Let $\mathcal N$ be the closed subscheme of $\mathcal P$ defined by $F,G$ and all $2\times2$-minors of the matrix
\[
 \begin{bmatrix}
  {\partial F}/{\partial x_1} & \cdots & {\partial F}/{\partial x_{n}}\\
  {\partial G}/{\partial x_1} & \cdots & {\partial G}/{\partial x_{n}}
 \end{bmatrix},
\]
and let $\pi_{\cN}:\mathcal N\to \cH\times\PP^{n-1}_\ZZ$ be the projection onto the first two factors.
\item[(iii)]
Let $\tilde{\mathcal N}$ be the closed subscheme of $\mathcal P$ defined by $G$ and its partial derivatives ${\partial G}/{\partial x_1}, \dotsc,{\partial G}/{\partial x_{n}}$, and let $\pi_{\tilde \cN}:\tilde{\mathcal N}\to \cH\times\PP^{n-1}_\ZZ$ be the projection onto the first two factors.
\end{itemize}
\end{notation}

\begin{notation}
\label{not:S(a,y,z)}
Suppose that $a \in \cH$ and write $k=k(a)$. Suppose that $y,z \in \PP^{n-1}_k$. Then we define
\begin{gather*}
\SS(a,y,z) = \pi_{\cM}^{-1}((a,y,z)), \\
\SS(a,y) = \pi_{\cN}^{-1}((a,y)), \\
\tilde\SS(a,y) = \pi_{\tilde\cN}^{-1}((a,y)), \\
\sigma(a,y) = \max \{\dim \SS(a,y), \dim \tilde\SS(a,y)\}.
\end{gather*}
Also, for each $s\in \{-1,0,1,\dotsc,n-1\}$, define
\[
\TT_s(a,y) = \{z \in \PP^{n-1}_k; \dim \SS(a,y,z) \geq s\}.
\]
\end{notation}

$\TT_s(a,y)$ is a closed subset of $\PP^{n-1}_k$, by Chevalley's theorem on upper semicontinuity of fibre dimension \cite[Cor 13.1.5]{EGAIV(III)}. 
Let us relate Notation \ref{not:S(a,y,z)} to Notation \ref{not:V_y}. In case $k$ is a perfect field, and $V$ is the hypersurface of $\PP^{n-1}_k$ corresponding to $a$, then the Jacobian criterion \cite[\S 4.2]{Liu} implies that $\SS(a,y,z) = \Sing(V_{\yy,\zz})$, $\SS(a,y) = \Sing(V_{\yy})$ and $\tilde\SS(a,y) = \Sing(\tilde V_{\yy})$. Thus, in this case, $\TT_s(a,y) = T_{s,\yy}(V)$.

\begin{notation}
Let $\cR_2$ be the set of $a \in \cH$ such that for all $y \in \PP^{n-1}_{k(a)}$ and all $s$, we have $\dim (\TT_{\sigma(a,y)+s+1}(a,y)) \leq n-2-s$.
\end{notation}

\begin{rem}
If $\aa \in \ZZ^{D+1}$ is the tuple of coefficients of $F \in \ZZ[x_1,\dotsc,x_n]$, then $F$ satisfies $(\mathrm R_2(p))$ if and only if $\aa \pmod p$ belongs to $\cR_2(\FF_p)$.
\end{rem}

Recall that a subset of a Noetherian topological space $X$ is constructible if and only if it can be written as a finite union of locally closed subsets of $X$ \cite[\textbf{0}, 9.1.7]{EGAIII(I)}.

Our key argument in deriving a criterion for $(\mathrm R_2(p))$ is the following fact, the proof of which uses a version of 'quantifier elimination' for schemes, developed by Chevalley and Grothendieck.

\begin{lemma}
\label{lem:constructible}
$\cR_2$ is a constructible subset of $\cH$.
\end{lemma}

\begin{proof}
Let $\cU_r$, for $r\in \{-1,0,1,\dotsc,n-1\}$, be the set of points $(a,y) \in \cH \times \PP^{n-1}_\ZZ$ such that $\sigma(a,y) \leq r$. Furthermore, for each pair $(s,u) \in \{-1,0,1,\dotsc,n-1\}^2$, let $\cQ_{s,u}$ be the set of points $(a,y) \in \cH \times \PP^{n-1}_\ZZ$ such that $\dim (\TT_s(a,y)) > u$. Using the semicontinuity theorem again, one sees that $\cU_r$ is open and $\cQ_{s,u}$ is closed.

Thus, the set
\[
\cS := \bigcup_{-1 \leq s,\sigma \leq n-1} \cU_{\sigma} \cap \cQ_{\sigma+s+1,n-2-s}
\]
is a constructible subset of $\cH \times \PP^{n-1}_\ZZ$. If $\pi:\cH \times \PP^{n-1}_\ZZ \to \cH$ denotes the projection onto the first factor, then by \cite[\textbf{IV}, 1.8.4]{EGAIV(I)}, $\pi(\cS)$ is a constructible subset of $\cH$. Since $\cR_2 = \cH \setminus \pi(\cS)$, and the family of constructible subsets is closed under complements, we have proved the lemma.
\end{proof}

As a consequence, we get the following result, which motivates this section. $\Vert F\Vert$ denotes the maximum of the absolute values of the coefficients of $F$.

\begin{cor}
\label{cor:veryverygood}
For each homogeneous polynomial $F \in \ZZ[x_1,\dotsc,x_n]$ of degree $d \geq 3$ defining a non-singular hypersurface in $\PP^{n-1}_\QQ$ (i.e. satisfying $(\mathrm R_0(0))$), the set of primes $\mathcal P(F)$ such that $F$ does not satisfy all of the conditions $(\mathrm R_0(p))$, $(\mathrm R_1(p))$ and $(\mathrm R_2(p))$, is finite. Furthermore, there are constants $C,\kappa$ depending only on $n$ and $d$, such that
\[
\prod_{p \in \cP(F)} p \leq C \Vert F \Vert^\kappa.
\] 
\end{cor}

\begin{proof}
Let $\cP_2(F)$ be the set of primes $p$ for which $F$ does not satisfy $(\mathrm R_2(p))$. Taking into account the results mentioned after Definition \ref{def:R_i(p)}, it suffices to prove that $\cP_2(F)$ is finite and that
\begin{equation}
\label{eq:prod_P2}
\prod_{p \in \cP_2(F)} p \leq C \Vert F \Vert^\kappa.
\end{equation} 
for constants $C,\kappa = O_{n,d}(1)$.

By Lemma \ref{lem:constructible}, we can write
\[
\cR_2 = \bigcup_{i=1}^k A_i \cap S_i,
\]
where the $A_i$ are open and the $S_i$ are closed. We may assume that $A_i = D_+(f_i)$ for homogeneous polynomials $f_i \in \ZZ[x_1,\dotsc,x_n]$ (see \cite[\S II.2]{Hartshorne}).

Suppose now that $F$ satisfies the hypotheses of the Corollary, and let $\aa \in \ZZ^{D+1}$ be the tuple of coefficients of $F$. Then we have $\aa \in \cR_2(\QQ)$, so there is an index $i$ such that $\aa \in S_i(\QQ)$ and $f_i(\aa) \neq 0$. For every $p \in \cP_2(F)$, we then have $p \mid f_i(\aa)$, since $\aa \in S_i(\QQ)$ implies $\aa \pmod p \in S_i(\FF_p)$. Thus, $\cP_2(F)$ is a finite set and
\[
\prod_{p \in \cP_2(F)} p \leq |f_i(\aa)| \leq \Vert f_j \Vert |\aa|^\kappa,
\]
where $\kappa = \deg f_j$.
\end{proof}


\section{Preliminary number theoretic results}
\label{sec:numbertheory}

We begin with some remarks on the results from the author's paper \cite{Marmon} that we will use.

\begin{rem}
\label{rem:Marmon}
 The error term \[D_{n+1}B^{s+1}q^{(n-r-s-2)/2}(B+q^{1/2})\] in \cite[Theorem 3.3]{Marmon} can be given by the simpler expression \[D_{n+1}B^{s+2}q^{(n-r-s-2)/2}.\] Indeed, if $q^{1/2}\gg B$, then one would  have $B^{s+2}q^{(n-r-s-2)/2}\gg B^{n-r}$, so that the theorem would be true by means of a trivial estimate, such as \cite[Lemma 3.1]{Marmon}.
\end{rem}

We shall in the proof use the weighted asymptotic formula mentioned in \cite[Remark 4.4]{Marmon}. Let us therefore state this result. Appealing to Remark \ref{rem:Marmon}, we may simplify the error term somewhat.

\begin{thm}
\label{thm:marmon}
Let $W:\RR^n \to [0,1]$ be an infinitely differentiable function supported on $[-2,2]^n$. Let $f_1,\dotsc,f_r$ be polynomials in $\ZZ[x_1,\dotsc,x_n]$ of degree at least 3, with leading forms $F_1,\dotsc,F_r$. Let
\[
 Z = \Proj \ZZ[x_1,\dotsc,x_n]/(F_1,\ldots,F_r)
\]
and suppose that $p$ and $q$ are primes, with $p \leq B \leq q$, such that both $Z_p$ and $Z_q$ are non-singular subvarieties of $\PP^{n-1}_{\FF_q}$ of dimension $n-1-r$. Then we have
\begin{multline*}
N_W(f_1,\dotsc,f_r,B,pq) 
-
p^{-r} q^{-r} N_W(0,B,pq) \\
\ll_{W,n,d,C}
B^{(n+1)/2} p^{-r/2} q^{(n-r-1)/4} + B^{(n+1)/2} p^{(n-2r)/2} q^{-1/4}\\
+ B^n p^{-(n+r-1)/2} q^{-r} + B^{n-C/2}p^{(C-r)/2}q^{-r/2}   
\end{multline*}
for any $C>0$, where $d=\max_i (\deg f_i)$.
\end{thm}

The following result is standard \cite[Lemma 3.1]{Marmon}.

\begin{prop}
\label{prop:trivial} Let $X=\Spec\mathbb{F}_{q}[x_{1},\ldots,x_{n}]/(f_{1},\ldots,f_{\rho})$
be a closed subscheme of $\mathbb{A}_{\mathbb{F}_{q}}^{n}$, and let
$d=\max_{i}(\deg f_{i})$. For any box 
\[
\sB=\left[a_{1}-b_{1},a_{1}+b_{1}\right]\times\ldots\times\left[a_{n}-b_{n},a_{n}+b_{n}\right],
\]
with $|b_{i}|\leq B$, containing at most one representative of each
congruence class modulo $q$, let $\sB_q$ be its image in $(\ZZ/q\ZZ)^n$. Then we have
\[
\#(\sB_q\cap X(\FF_q))\ll_{n,\rho,d}B^{\dim X}.
\]
\end{prop}

\begin{rem*}
The dependence on $\rho$ can be eliminated - one can show \cite[Cor V.1.5]{Kunz} that there is an ideal generated by at most $n$ elements whose radical equals the radical of $(f_{1},\ldots,f_{\rho})$.
\end{rem*}

%
%

The following asymptotic formula for the number of rational points on a complete intersection, due to Hooley \cite{Hooley}, is a consequence of the Weil conjectures \cite{Deligne}. The version below is proved in \cite[Lemma 3.2]{Marmon}.

\begin{prop}
 \label{hooley-deligne}
 Let $f_1,\ldots,f_r$ be polynomials in $\FF_q[x_1,\ldots,x_n]$ with leading forms $F_1,\ldots,F_r$, respectively. Let
 \begin{align*}
  X&=\Spec \FF_q[x_1,\ldots,x_n]/(f_1,\ldots,f_r),\\
  Z&=\Proj \FF_q[x_1,\ldots,x_n]/(F_1,\ldots,F_r).
 \end{align*}
 Suppose that $\dim Z=n-1-r$ and let $s=\dim \Sing Z$. Then
 \[
  \#X(\FF_q)=q^{n-r} + O_{n,d}\left(q^{(n-r+2+s)/2}\right),
 \]
 where $d=\max_i(\deg F_i)$.
\end{prop}

The following result is a simple exercise in Poisson summation. The argument appears in \cite{Heath-Brown}.

\begin{lemma}
\label{lem:poisson}
 Let $\phi:\RR^n\to\RR$ be an infinitely differentiable function supported in the box $[-M,M]^n$, and let $D_k$, for $k=0,1,2,3,\dotsc$, be the maximum over $\RR^n$ of all partial derivatives of $\phi$ of order $k$. Let $a$ and $B$ be real numbers such that $B\geq1$ and $1\leq a\leq B$. Then we have
 \begin{multline*}
  \sum_{\xx\in\ZZ^n} \phi\left(\frac{1}{B}\xx\right) \sum_{\yy\in\ZZ^n} \phi\left(\frac{1}{B}(\xx+a\yy)\right) \\
= a^{-n}\left(\sum_{\xx\in\ZZ^n} \phi\left(\frac{1}{B}\xx\right)\right)^2 + O_{n,M,k}\left(D_0 D_k B^{2n-k}a^{-n+k}\right)\\
 + O_{n,M,k} \left(D_k^2 B^{2(n-k)}a^{-n+k}\right).
 \end{multline*}
for any $k\in\ZZ_{\geq 0}$.
\end{lemma}

\begin{proof}
Since $\phi$ is infinitely differentiable and compactly supported, we have for the Fourier transform $\hat\phi$ the estimate
 \begin{equation}
 \label{eq:phihat}
  \hat\phi(\xxi) \ll_{n,M,k} D_k{\vert \xxi \vert}^{-k},
 \end{equation}
The function $\Phi(\xx)=\phi((1/B)\xx)$ has the Fourier transform $\hat\Phi(\xxi)=B^n \hat\phi(B\xxi)$. Thus, by Poisson's summation formula and \eqref{eq:phihat}, we get
\begin{equation}
 \label{eq:poisson_x}
 \begin{aligned}
  \sum_{\xx\in\ZZ^n} \phi\left(\frac{1}{B}\xx\right) 
  &=
  B^n \sum_{\xxi\in\ZZ^n} \hat\phi\left(B\xxi\right) \\
  &=
  B^n \hat\phi(\0) + O_{n,M,k}(D_k B^{n-k}).
 \end{aligned}
\end{equation}
For fixed $\xx$, put $\psi(\yy)=\phi((1/B)(\xx+a\yy))$. Then 
\[
\hat\psi(\eeta)=\left(\frac B a \right)^n \exp(-2\pi i a^{-1} \xx \cdot \eeta) \hat\phi \left( \frac B a \eeta \right).
\]
By Poisson's summation formula and \eqref{eq:poisson_x} we calculate
\begin{equation}
\label{eq:poisson_y}
\begin{aligned}
 \sum_{\yy\in\ZZ^n} \phi\left(\frac{1}{B}(\xx+a\yy)\right)
 &=
 \left(\frac B a \right)^n  \sum_{\eeta\in\ZZ^n} \exp(-2\pi i a^{-1} \xx \cdot \eeta) \hat\phi \left( \frac B a \eeta \right)\\
 &=
 a^{-n} \left( B^n \hat\phi(\mathbf{0}) + O_{n,M,k}(D_k B^{n-k} a^k) \right)\\
 &=
 a^{-n} \sum_{\vv\in\ZZ^n} \phi\left(\frac{1}{B}\vv\right) + O_{n,M,k}(D_k B^{n-k} a^{-n+k}).
\end{aligned}
\end{equation}
Multiplying \eqref{eq:poisson_y} by $\phi\left(\frac{1}{B}\xx\right)$, summing over $\xx \in \ZZ^n$ and using \eqref{eq:poisson_x} and \eqref{eq:phihat} we get the desired formula.
\end{proof}


\section{The density of solutions to $f(\xx) \equiv 0 \pmod {\pi pq}$}
\label{sec:technicalheart}

This subsection constitutes the technical heart of the proof of Theorem \ref{thm:main}. Let $n\geq 5$, and let $f$ be a polynomial in $\ZZ[x_1,\ldots,x_n]$ of degree $d\geq4$, with leading form $F$. Let $W$ be the infinitely differentiable weight function in \eqref{eq:W}.

\begin{notation}[Differenced polynomials]
\label{not:Diff}
For any $\yy \in \ZZ^n$, define the polynomial $f^{\yy} \in \ZZ[x_1,\dotsc,x_n]$ by
\[
f^{\yy}(\xx) = f(\xx+\yy)-f(\xx).
\]
For any pair $(\yy,\zz) \in \ZZ^n \times \ZZ^n$, define 
\[
f^{\yy,\zz}(\xx) = f(\xx+\yy+\zz) - f(\xx+\yy) - f(\xx+\zz) + f(\xx).
\]
Furthermore, let
\[
F^\yy(\xx) = \yy \cdot \nabla F(\xx) = y_1\frac{\partial F}{\partial x_1}+\dotsb+y_n\frac{\partial F}{\partial x_n}
\]
and
\[
F^{\yy,\zz}(\xx)=(\Hess (F))\yy\cdot \zz=\sum_{1\leq i,j\leq n}\frac{\partial^2 F}{\partial x_i\partial x_j} y_i z_j
\]
as in Section \ref{sec:geometry}.

For any prime $p$, define the schemes
\begin{gather*}
Z_{p,\yy} = \Proj \ZZ[x_1,\dotsc,x_n]/(p,F,F^\yy), \\
\tilde Z_{p,\yy} = \Proj \ZZ[x_1,\dotsc,x_n]/(p,F^\yy), \\
Z_{p,\yy,\zz} = \Proj \ZZ[x_1,\dotsc,x_n]/(p,F,F^\yy,F^{\yy,\zz}).
\end{gather*}
Put
\begin{gather*}
s_p(\yy) = \dim \Sing(Z_{p,\yy}), \\
\tilde s_p(\yy) = \dim \Sing (\tilde Z_{p,\yy}), \\
\sigma_p(\yy) = \max \{s_p(\yy),\tilde s_p(\yy)\}, \\
s_p(\yy,\zz) = \dim  \Sing (Z_{p,\yy,\zz}).
\end{gather*}
If $Z_p = \Proj \ZZ[x_1,\dotsc,x_n]/(p,F)$, we have $s_p(\yy) = s_{\yy_p}(Z_p)$, $\tilde s_p(\yy) = \tilde s_{\yy_p}(Z_p)$, $\sigma_p(\yy) = \sigma_{\yy_p}(Z_p)$ (cf. Notation \ref{not:V_y}), where $\yy_p$ denotes the image of $\yy$ under the natural map $\ZZ^n \to (\ZZ/p\ZZ)^n$ .
\end{notation}

\begin{notation}['Differenced' weight functions]
For any $\aa \in \RR^n$, let the infinitely differentiable function $W_\aa$ be given by
\[
W_\aa(\xx) = W(\xx)W(\xx + \aa).
\]
Note that $W_\aa$ vanishes identically if $|\aa| \geq 4$.
Also define, for any pair $(\aa,\aa') \in \RR^n \times \RR^n$, the function
\begin{align*}
W_{\aa,\aa'}(\xx) &= W_\aa(\xx)W_\aa(\xx + \aa') \\
&= W(\xx)W(\xx+\aa)W(\xx+\aa')W(\xx+\aa+\aa').
\end{align*}
\end{notation}

Suppose that we are given three different prime numbers $\pi,p,q$, with $\pi,p\leq B < q/4$, such that $F$ satisfies 
\begin{equation}
\label{eq:R_i()}
\begin{gathered}
(\mathrm R_0(\pi)), \\
(\mathrm R_0(p)), (\mathrm R_1(p)),\\
(\mathrm R_0(q)), (\mathrm R_1(q)) \text{ and }(\mathrm R_2(q))
\end{gathered}
\end{equation}
(as defined in Definition \ref{def:R_i(p)}). We shall later prove the existence of suitable primes $\pi,p,q$.

\begin{lemma}
\label{lem:main}
  
Under the hypotheses above, we have the following results:
\begin{enumerate}

 \item \label{ml1}
 Put
 \[
 K=\pi^{-n}p^{-1}q^{-1}N_W(0,B,\pi pq). 
 \]
 Then
 \begin{equation}
 \label{eq:ml1}
 \begin{aligned}
 N_W(f,B,\pi pq)&=(\pi pq)^{-1}N_W(0,B,\pi pq)+O\left(\pi^{(n-1)/2}\Sigma^{1/2}\right)\\
                &\hphantom{=} +O\left(B^n\pi^{-n/2}p^{-1}q^{-1}\right),\text{ where}
 \end{aligned}
 \end{equation}  
\[ 
\Sigma=\sum_{\uu\in\FF_\pi^n}\left( \sum_{\substack{\xx\equiv\uu\smod{\pi}\\
pq\mid f(\xx)}}
W\left(\frac{1}{B}\xx\right)-K\right)^2.                                                                   \]          
\item \label{ml2}
For any $\yy \in \ZZ^n$, put
\[
 \Delta(\yy)=\sum_{\substack{\xx\in\ZZ^n\\pq\mid f(\xx) \\pq \mid f(\xx+\pi \yy)}} W_{\pi\yy}\left(\frac{1}{B}\xx\right) - p^{-2}q^{-2} \sum_{\xx\in\ZZ^n} W_{\pi\yy}\left(\frac{1}{B}\xx\right).
\]
Then
 \begin{equation}
 \label{eq:ml2}
  \Sigma = \sum_{|\yy|\leq 4B/\pi}\Delta(\yy) + \mathcal{E}_1.
 \end{equation}
 
\item \label{ml3}
Suppose that $\yy\neq 0$. For any $\zz \in \ZZ^n$, put
\[
 \Delta(\yy,\zz):=\sum_{\substack{\xx\in\ZZ^n\\q\mid f(\xx)\\q \mid f^{p\zz}(\xx)\\
q\mid f^{\pi\yy,p\zz}(\xx)}} W_{\pi\yy,p\zz}\left(\frac{1}{B}\xx\right) - q^{-3}\sum_{\xx\in\ZZ^n}W_{\pi\yy,p\zz}\left(\frac{1}{B}\xx\right).
\]
Then 
\begin{equation}
\label{eq:ml3}
\Delta(\yy)=p^{(n-2)/2}\left( \sum_{|\zz|\leq 4B/p} \Delta(\yy,\zz)\right)^{1/2} + \mathcal{E}_2(\yy) + \mathcal{E}_3. 
\end{equation} 

\item \label{ml4}
Furthermore, we have
\begin{equation}
\label{eq:Delta(0)}
\Delta(\0) \ll B^n p^{-1} q^{-1} + \mathcal{E}_0.
\end{equation}
\end{enumerate}

All the implied constants depend only on $n$ and $d$, unless otherwise specified. The error terms $\mathcal E_i$ satisfy the following estimates:
\begin{align*}
\mathcal{E}_0 &\ll B^{(n+1)/2}p^{-1/2}q^{(n-2)/4} + B^{(n+1)/2}p^{(n-2)/2}q^{-1/4} + B^{n}p^{-n/2}q^{-1}. \\
\mathcal{E}_1 &\ll_C B^{(3n+1)/2}\pi^{-n}p^{-3/2}q^{(n-6)/4} + B^{(3n+1)/2}\pi^{-n}p^{(n-4)/2}q^{-5/4} \\ 	&\hphantom{\ll_C} + B^{2n}\pi^{-n}p^{-(n+2)/2}q^{-2} + B^{2n-C}\pi^{-n+C}p^{-2}q^{-2},\\
&\hphantom{\ll_C} \text{for any } C>0. \\
\mathcal{E}_2(\yy) &\ll B^n p^{(n-s_p(\yy))/2} q^{-2}. \\
\mathcal{E}_3 &\ll_C B^{(n+1)/2} p^{-1} q^{(n-6)/4} + B^{n-C} p^{-1+C} q^{-3/2},\quad \text{for any } C>0.
\end{align*}

\end{lemma}
 
\begin{proof}
Starting from the definition of $N_W(f,B,\pi pq)$, we write 
\[
\begin{split}
N_W(f,B,\pi pq)
               &=\sum_{\substack{\uu\in\FF_\pi^n \\ f_\pi(\uu)=0}}\left( \sum_{\substack{\xx\equiv\uu\smod{\pi}\\ pq\mid f(\xx)}}W\left(\frac{1}{B}\xx\right)-K\right) +K \sum_{\substack{\uu\in\FF_\pi^n \\ f_\pi(\uu)=0}}1\\
	       &=S+K\left(\pi^{n-1}+O(\pi^{n/2})\right),                                                      
	       \end{split}
\]
where
\[
S=\sum_{\substack{\uu\in\FF_\pi^n \\ f_\pi(\uu)=0}}\left( \sum_{\substack{\xx\equiv\uu\smod{\pi}\\pq\mid f(\xx)}}W\left(\frac{1}{B}\xx\right)-K\right),
\]
and $f_\pi \in \ZZ/\pi\ZZ[x_1,\dotsc,x_n]$ is the image of $f$ under the natural homomorphism $\ZZ[x_1,\dotsc,x_n] \to \ZZ/\pi\ZZ[x_1,\dotsc,x_n]$.
Here we have used the property $(\mathrm R_0(\pi))$, applying Proposition \ref{hooley-deligne} to the hypersurface defined by $f_\pi(\uu)=0$. By Cauchy's inequality,
\[
 S^2\ll \pi^{n-1}\sum_{\uu\in\FF_\pi^n}\left( \sum_{\substack{ \xx\equiv\uu\smod{\pi}\\pq\mid f(\xx)}}W\left(\frac{1}{B}\xx\right)-K\right)^2,
\]
so we have
\[
 N_W(f,B,\pi pq) = \pi^{n-1}K+O\left(\pi^{(n-1)/2}\Sigma^{1/2}\right)+O\left(B^n\pi^{-n/2}p^{-1}q^{-1}\right),
\]
and (\ref{ml1}) is proved. Now,
\[
\begin{split}
 \Sigma&=\sum_{\uu\in\FF_\pi^n}\left( \sum_{\substack{\xx\equiv\uu\smod{\pi}\\
 pq \mid f(\xx)}}W\left(\frac{1}{B}\xx\right)-K\right)^2\\
       &=\sum_{\uu\in\FF_\pi^n}\left( \sum_{\substack{\xx\equiv\uu\smod{\pi}\\
       pq\mid f(\xx)=0}}W\left(\frac{1}{B}\xx\right)\right)^2 - 2K N_W(f,B,pq) + \pi^n K^2.
 \end{split}
\]
Using Theorem \ref{thm:marmon}, along with the properties $(\mathrm R_0(p))$ and $(\mathrm R_0(q))$, we have
\begin{equation}
\label{eq:N_W(f,B,pq)}
N_W(f,B,pq)=\pi^n K +\mathcal{E}_0, 
\end{equation}
where
\begin{multline*}
 \mathcal{E}_0 \ll B^{(n+1)/2}p^{-1/2}q^{(n-2)/4} + B^{(n+1)/2}p^{(n-2)/2}q^{-1/4} + B^{n}p^{-n/2}q^{-1}.
\end{multline*}
(The last error term in Theorem \ref{thm:marmon} becomes negligible upon taking $C \geq n-1$.)
We conclude that
\[
 \Sigma = \sum_{\uu\in\FF_\pi^n}\left( \sum_{\substack{\xx\equiv\uu\smod{\pi}\\
 pq \mid f(\xx)}}W\left(\frac{1}{B}\xx\right)\right)^2 - \pi^n K^2 + K \mathcal{E}_0.
\]
Introducing a new variable $\yy$, we expand the sum of squares as a double sum 
\begin{equation*}
\label{eq:diff1}
 \begin{split}
  \sum_{\uu\in\FF_\pi^n}\left( \sum_{\substack{\xx\equiv\uu\smod{\pi}\\
  pq \mid f(\xx)}}W\left(\frac{1}{B}\xx\right)\right)^2 
  &= \sum_{\substack{\xx\in\ZZ^n\\pq \mid f(\xx)}} W\left(\frac{1}{B}\xx\right) \hspace{-1ex} \sum_{\substack{\yy\in\ZZ^n\\pq\mid f(\xx+\pi\yy)}} \hspace{-1ex} W\left(\frac{1}{B}(\xx+\pi\yy)\right)\\
  &=\sum_{|\yy|\leq 4B/\pi}\sum_{\substack{\xx\in\ZZ^n\\
  pq\mid f(\xx)\\
  pq \mid f(\xx+\pi\yy)}}W_{\pi\yy}\left(\frac{1}{B}\xx\right).
 \end{split}
\end{equation*}
Recalling the definition of $\Delta(\yy)$ above, we have
\[
 \Sigma = \sum_{|\yy|\leq 4B/\pi}\Delta(\yy) + p^{-2}q^{-2} \sum_{\yy\in\ZZ^n}\sum_{\xx\in\ZZ^n} W_{\pi\yy}\left(\frac{1}{B}\xx\right) - \pi^n K^2 + K \mathcal{E}_0.
\]
By Lemma \ref{lem:poisson}, however,
\[
 p^{-2}q^{-2} \sum_{\yy\in\ZZ^n}\sum_{\xx\in\ZZ^n} W_{\pi\yy}\left(\frac{1}{B}\xx\right) - \pi^n K^2 \ll_C B^{2n-C}\pi^{-n+C}p^{-2}q^{-2},
\]
so letting 
\[
\mathcal{E}_1 = K \mathcal{E}_0 + B^{2n-C}\pi^{-n+C}p^{-2}q^{-2},
\]
we have proved (\ref{ml2}).

\begin{rem}
\label{rem:salberger's idea}
$\Delta(\yy)$ measures the deviation of the weighted number of solutions to the two simultaneous congruences $pq\mid f(\xx),\ pq \mid f(\xx+\pi\yy)$ from its expected value. Unlike in the papers by Heath-Brown \cite{Heath-Brown} and the author \cite{Marmon}, we keep both congruence conditions in $\Delta(\yy)$ instead of using just the differenced polynomial $f(\xx+\pi\yy)-f(\xx)$. This is the approach introduced by Salberger \cite{Salberger-Integral_points}.
\end{rem}

By Remark \ref{rem:finite} and the properties $(\mathrm R_0(p))$ and $(\mathrm R_0(q))$, neither of $F^{\yy}_q$ and  $F^{\yy}_p$ is identically zero. This means that $f^{\pi\yy}$ is a polynomial of degree $d-1$ with leading form $\pi F^\yy$, and moreover 
\[
\dim Z_{q,\yy}=\dim Z_{p,\yy}=n-3.
\] 
Let
\[
 X_{\yy}=\Spec \ZZ[x_1,\dotsc,x_n]/(f,f^{\pi\yy}).
\]

Now we write
\begin{equation}
 \label{eq:delta(y)}
 \Delta(\yy)=S(\yy) +\mathcal{E}_2(\yy),
\end{equation}
where we have defined
\[
 S(\yy)=\sum_{\substack{\vv\in\FF_p^n\\f_p(\vv)=f^{\pi\yy}_p(\vv)=0}} \left( \sum_{\substack{\xx\equiv\vv\smod{p}\\q \mid f(\xx)\\q \mid f^{\pi\yy}(\xx)}} W_{\pi\yy}\left(\frac{1}{B}\xx\right) - K(\yy) \right)
\]
and
\begin{align*}
 \mathcal{E}_2(\yy) &= \sum_{\substack{\vv\in\FF_p^n\\f_p(\vv)=f^{\pi\yy}_p(\vv)=0}} K(\yy) - p^{-2}q^{-2}\sum_{\xx\in\ZZ^n} W_{\pi\yy}\left(\frac{1}{B}\xx\right),
\end{align*}
with
\[
 K(\yy):=p^{-n}q^{-2} \sum_{\xx\in\ZZ^n}W_{\pi\yy}\left(\frac{1}{B}\xx\right).
\]
But then
\[
 \mathcal{E}_2(\yy) = K(\yy) \left(\#X_{\yy}(\FF_p) - p^{n-2}\right),
\]
and by Proposition \ref{hooley-deligne} we have $\#X_{\yy}(\FF_p)=p^{n-2}+O\left(p^{(n+s_p(\yy))/2}\right)$, yielding 
\[
 \mathcal{E}_2(\yy) \ll B^n p^{-(n-s_p(\yy))/2} q^{-2}.
\]
Thus we turn now to $S(\yy)$. We again apply Cauchy's inequality, using Proposition \ref{prop:trivial} to estimate the number of $\FF_p$-points on $X_\yy$. Thus we get
\begin{equation}
\label{eq:S(y)}
 S(\yy)^2 \ll p^{n-2} \Sigma(\yy), 
\end{equation}
where
\[
 \Sigma(\yy) = \sum_{\vv\in\FF_p^n} \left( \sum_{\substack{\xx\equiv\vv\smod{p}\\q \mid f(\xx)\\q \mid f^{\pi\yy}(\xx)}} W_{\pi\yy}\left(\frac{1}{B}\xx\right) - K(\yy) \right)^2.
\]
\begin{rem}
\label{rem:second diff}
In this second differencing step, our approach is intermediate between that of Heath-Brown and that of Salberger. Indeed, we shall complete the sum (mod $q$), as in Heath-Brown's original argument \cite{Heath-Brown}, with respect to one of the two polynomials involved. This leads us to consider the closed subscheme defined by the three polynomials $f,f^{p\zz},f^{\pi\yy,p\zz}$ rather than the one defined by the four polynomials $f,f^{\pi\yy},f^{p\zz},f^{\pi\yy,p\zz}$. The reason is that the geometric results of \cite{Salberger-Integral_points} extend more readily in the former case.
\end{rem}
We have
\begin{equation}
\label{eq:Sigma(y)}
 \Sigma(\yy) 
 \leq
 \sum_{\vv\in\FF_p^n} \sum_{a=1}^q \left( \sum_{\substack{\xx\equiv\vv\smod{p}\\q \mid f(\xx)\\f^{\pi\yy}(\xx)\equiv a\smod{q}}} W_{\pi\yy}\left(\frac{1}{B}\xx\right)-K(\yy)\right)^2.
\end{equation}
Denote by $\Sigma'(\yy)$ the right hand side of (\ref{eq:Sigma(y)}). Expanding the square, we get
\begin{multline*}
\Sigma'(\yy) = \sum_{\vv\in\FF_p^n} \sum_{a=1}^q \left( \sum_{\substack{\xx\equiv\vv\smod{p}\\q \mid f(\xx)\\f^{\pi\yy}(\xx)\equiv a\smod{q}}} W_{\pi\yy}\left(\frac{1}{B}\xx\right)\right)^2 \\
-2K(\yy) N_{W_{\pi\yy}}(f,B,q) + p^n q K(\yy)^2.
\end{multline*}
By Remark \ref{rem:Marmon} we have the estimate
\[
 N_{W_{\pi\yy}}(f,B,q)=p^n q K(\yy)+O\left(Bq^{(n-2)/2}\right), 
\]
insertion of which yields
\begin{equation}
\label{eq:Sigma'(y)}
\begin{aligned}
\Sigma'(\yy) &= \sum_{\vv\in\FF_p^n} \sum_{a=1}^q \left( \sum_{\substack{\xx\equiv\vv\smod{p}\\q \mid f(\xx)\\f^{\pi\yy}(\xx)\equiv a\smod{q}}} W_{\pi\yy}\left(\frac{1}{B}\xx\right)\right)^2  - p^n q K(\yy)^2 \\
             &\hphantom{=} + O \left( B^{n+1} p^{-n} q^{(n-6)/2} \right).
\end{aligned}
\end{equation} 
As before, we proceed to expand the sum of squares as a double sum, introducing a third variable $\zz$:
\begin{align*}
\sum_{\vv\in\FF_p^n} \sum_{a=1}^q \left( \sum_{\substack{\xx\equiv\vv\smod{p}\\q \mid f(\xx)\\f^{\pi\yy}(\xx)\equiv a\smod{q}}} W_{\pi\yy}\left(\frac{1}{B}\xx\right)\right)^2 
=
\sum_{|\zz| \leq 4B/p} \sum_{\substack{\xx\in\ZZ^n\\q\mid f(\xx)\\q \mid f^{p\zz}(\xx)\\q\mid f^{\pi\yy,p\zz}(\xx)}} 
W_{\pi\yy,p\zz} \left(\frac 1 B \xx \right),
\end{align*}
and then we compare the sum over $\xx$ to its expected value $\Delta(\yy,\zz)$. Another application of Lemma \ref{lem:poisson} yields that
\[
 q^{-3}\sum_{|\zz|\leq 4B/p}\sum_{\xx\in\ZZ^n} W_{\pi\yy,p\zz}\left(\frac{1}{B}\xx\right) -  p^n q K(\yy)^2 \ll_C B^{2n-C}p^{-n+C}q^{-3},
\]
so it follows, in view of \eqref{eq:delta(y)}, (\ref{eq:S(y)}), (\ref{eq:Sigma(y)}) and (\ref{eq:Sigma'(y)}), that
\begin{equation} 
\Delta(\yy)\ll p^{(n-2)/2}\left( \sum_{|\zz|\leq 4B/p}\Delta(\yy,\zz) \right)^{1/2} + \mathcal{E}_2(\yy)+ \mathcal{E}_3, 
\end{equation}
where 
\[
\mathcal{E}_3 \ll_C B^{(n+1)/2} p^{-1} q^{(n-6)/4} + B^{n-C} p^{-1+C} q^{-3/2},\quad \text{for any } C>0.
\]
This proves (\ref{ml3}). 

Finally, we have
\[
\Delta(\mathbf 0) = N_{W_{\pi\yy}}(f,B,pq) - p^{-2}q^{-2} \sum_{\xx\in\ZZ^n} W_{\pi\yy}\left(\frac{1}{B}\xx\right),
\]
so arguing as in \eqref{eq:N_W(f,B,pq)}, we get (\ref{ml4}). 
\end{proof}

By (\ref{eq:ml1}),(\ref{eq:ml2}) and (\ref{eq:ml3}) we are led to evaluate the quantity
\begin{equation}
\label{eq:E4}
 \mathcal{E}_4 := \pi^{(n-1)/2} p^{(n-2)/4} \left( \sum_{\yy \in \ZZ^n \setminus \{\0\}} \left( \sum_{\zz \in \ZZ^n} |\Delta(\yy,\zz)| \right)^{1/2} \right)^{1/2},
\end{equation}
which will be the strongest competitor to the main term in \eqref{eq:ml1}. We shall derive an estimate for $\cE_4$, subject to additional hypotheses. We maintain the convention that implied constants depend only on $n$ and $d$, unless
otherwise specified.

\begin{lemma}
\label{lem:E4}
In addition to the hypotheses preceding Lemma \ref{lem:main}, suppose that $B \geq q^{1/2}$ and $q \nmid (d-1)$. Then we have the estimate
\begin{equation}
\label{eq:lem:E4}
\begin{split}
\cE_4 &\ll B^{(3n+1)/4} \pi^{-1/2} p^{-1/2} q^{(n-4)/8} + B^{(3n+1)/4} \pi^{(n-5)/4} p^{-1/2} q^{-1/8} \\
&\hphantom{\ll} + B^{(3n+1)/4} \pi^{-1/2} p^{(n-5)/4} q^{-1/8} + B^{3n/4} \pi^{(n-3)/4} p^{-1/4} q^{-1/2} \\
&\hphantom{\ll} + B^{(n+1)/2} \pi^{(n-3)/4} p^{-1/4} q^{(n-4)/8} + B^{3n/4} \pi^{(n-4)/4} q^{-1/2}  \\
&\hphantom{\ll} + B^{(2n+3)/4} \pi^{(n-4)/4} q^{(n-5)/8} +B^{(2n+3)/4} \pi^{(n-4)/4} p^{(n-4)/4} q^{-1/8} \\
&\hphantom{\ll} + B^{3n/4}\pi^{-1/2} p^{(n-2)/4} q^{-1/4} + B^{(2n+1)/4} \pi^{-1/2} p^{(n-2)/4} q^{(n-2)/8}.
\end{split}
\end{equation}
\end{lemma}

\begin{rem*}
If we would remove the hypothesis $B \geq q^{1/2}$, then we would get even more terms in \eqref{eq:lem:E4}. The hypothesis $q \nmid (d-1)$ ensures that Lemma \ref{lem:T_s-special}(ii) is applicable.
\end{rem*}

\begin{proof}
We wish to switch the order of summation in \eqref{eq:E4}. Thus we apply H\"older's inequality \cite[Theorem 11]{Inequalities} to get 
\[
 \sum_{\yy \in \ZZ^n \setminus \{\0\}} \left( \sum_{\zz \in \ZZ^n} |\Delta(\yy,\zz)| \right)^{1/2} 
 \ll 
 \left( \frac{B}{\pi} \right)^{n/2} \left( \sum_{\yy \in \ZZ^n \setminus \{\0\}} \sum_{\zz \in \ZZ^n} |\Delta(\yy,\zz)| \right)^{1/2}.
\]
Here we have used the fact that $\Delta(\yy,\zz)$ vanishes identically for $|\yy| \geq 4B/\pi$.
\eqref{eq:E4} transforms into
\begin{equation}
\label{eq:E4'}
 \mathcal{E}_4 \ll B^{n/4} \pi^{(n-2)/4} p^{(n-2)/4} \left( \sum_{(\yy,\zz) \in \cB} |\Delta(\yy,\zz)| \right)^{1/4},
\end{equation}
where the domain of summation is defined by
\[
\cB = \left\{(\yy,\zz) \in \ZZ^n \times \ZZ^n; |\yy| \leq \frac{4B}{\pi}, |\zz| \leq \frac{4B}{p}, \yy \neq \mathbf 0 \right\}. 
\]
To calculate $\cE_4$, we shall partition $\cB$ into three subsets
\begin{gather*}
\cB_1 = \{(\yy,\zz) \in \cB;\ \dim Z_{q,\yy,\zz} = n-4\}, \\
\cB_2 = \{(\yy,\zz) \in \cB;\ \dim Z_{q,\yy,\zz} > n-4,\ \zz \neq \0\}, \\
\cB_3 = \{(\yy,\zz) \in \cB;\ \zz = \0\},
\end{gather*}
and let
\[
 \mathcal{E}_{4,i} = B^{n/4} \pi^{(n-2)/4} p^{(n-2)/4} \left( \sum_{(\yy,\zz) \in \cB_i} |\Delta(\yy,\zz)| \right)^{1/4}
\]
for $i \in \{1,2,3\}$.

We start with $\cE_{4,1}$. Let us partition $\cB_1$ even further into subsets (some of them possibly empty)
\[
\cB_{1,\sigma,s} = \{(\yy,\zz) \in \cB_1; \sigma_q(\zz) = \sigma, s_q(\yy,\zz) = s\},
\]
where $s \in \{-1,\dotsc,n-4\}$, $\sigma \in \{-1,\dotsc,n-1\}$ and $\sigma_q(\cdot)$ and $s_q(\cdot,\cdot)$ are as defined in Notation \ref{not:Diff}. Using \cite[Thm. 3.3]{Marmon} (combined with Remark \ref{rem:Marmon} as usual) we have the estimate
\begin{equation}
\label{eq:Delta(y,z)}
|\Delta(\yy,\zz)| \ll B^{s+2} q^{(n-5-s)/2}
\end{equation}
for $(\yy,\zz) \in \cB_{1,\sigma,s}$. Recall that $F$ satisfies $(\R_1(q))$ and $(\R_2(q))$. This implies that
\[
\dim T_{\sigma}(Z_q) \leq n-2-\sigma,
\]
for $\sigma \in \{-1,\dotsc,n-1\}$. Furthermore, if $\zz \in \ZZ^n$ satisfies $\sigma_q(\zz) = \sigma$, then 
\[
\dim T_{s,\zz}(Z_q) \leq n-1-s+\sigma
\]
for $s \geq \sigma$. Applying Proposition \ref{prop:trivial}, we get
\begin{equation}
\label{eq:B_{1,sigma,s}}
\# \cB_{1,\sigma,s} \ll \begin{cases}
\left( \dfrac{B}{p}\right)^{n-1-\sigma} \left( \dfrac{B}{\pi}\right)^{n-s+\sigma} &\text{ if } s \geq \sigma,\\
\left( \dfrac{B}{p}\right)^{n-1-\sigma} \left( \dfrac{B}{\pi}\right)^{n}  &\text{ if } s < \sigma.
\end{cases}
\end{equation}
Combining \eqref{eq:Delta(y,z)} and \eqref{eq:B_{1,sigma,s}}, we have
\[
\sum_{(\yy,\zz) \in \cB_{1,\sigma,s}} |\Delta(\yy,\zz)| \ll U_{\sigma,s},
\]
where
\[
U_{\sigma,s}:=
\begin{cases}
B^{2n+1} \pi^{-n+s-\sigma} p^{-n+1+\sigma} q^{(n-5-s)/2} &\text{ if } s \geq \sigma,\\
B^{2n+s-\sigma+1} \pi^{-n} p^{-n+1+\sigma} q^{(n-5-s)/2} &\text{ if } s < \sigma.
\end{cases} 
\]
By carefully examining the relations between the quantities $U_{\sigma,s}$, one sees that
\begin{align*}
\sum_{(\yy,\zz) \in \cB_{1}} |\Delta(\yy,\zz)| &\ll \sum_{\sigma,s} U_{\sigma,s} \ll U_{-1,-1} + U_{-1,n-4} + U_{n-4,n-4} \\
&= B^{2n+1} \pi^{-n} p^{-n} q^{(n-4)/2} + B^{2n+1} \pi^{-3}p^{-n}q^{-1/2} \\
&\hphantom{=} + B^{2n+1} \pi^{-n} p^{-3} q^{-1/2}.
\end{align*}
We conclude that
\begin{equation}
\label{eq:E_4,1}
\begin{split}
\cE_{4,1} &\ll B^{(3n+1)/4} \pi^{-1/2} p^{-1/2} q^{(n-4)/8} + B^{(3n+1)/4} \pi^{(n-5)/4} p^{-1/2} q^{-1/8} \\
&\hphantom{=} + B^{(3n+1)/4} \pi^{-1/2} p^{(n-5)/4} q^{-1/8}. 
\end{split}
\end{equation}

Next, we turn our attention to $\cE_{4,2}$. For $(\yy,\zz) \in \cB_2$, we cannot apply the case $r=3$ of Remark \ref{rem:Marmon} to estimate $|\Delta(\yy,\zz)|$ as in \eqref{eq:Delta(y,z)}, since $\dim Z_{q,\yy,\zz} > n-4$. However, since we assume that $F$ satisfies the property $\R_0(q)$, we do know, by Lemma \ref{lem:dimV_y}, that $\dim Z_{q,\zz} = n-3$. Thus, it follows from the case $r=2$ of Remark \ref{rem:Marmon} that
\begin{equation}
\label{eq:Delta(y,z)_2}
\begin{split}
|\Delta(\yy,\zz)| &\leq N_{W_{\pi\yy,p\zz}}(f,f^{p\zz},B,q) + O(B^n q^{-3}) \\
&= q^{-2}N_{W_{\pi\yy,p\zz}}(0,B,q) + O\left(B^{\sigma_q(\zz)+2} q^{(n-4-\sigma_q(\zz))/2}\right) \\
&\hphantom{=} + O(B^n q^{-3})\\
&\ll B^n q^{-2} +B^{\sigma_q(\zz)+2} q^{(n-4-\sigma_q(\zz))/2}.
\end{split}
\end{equation}
Here we have used the fact that $\sigma_q(\zz) \geq s_q(\zz)$, so that $(Bq^{-1/2})^{\sigma_q(\zz)} \geq (Bq^{-1/2})^{s_q(\zz)}$.

We shall partition $\cB_2$ into subsets
\[
\cB_{2,\sigma} = \{(\yy,\zz) \in \cB_2;\ \sigma_q(\zz) = \sigma\}
\]
where $\sigma \in \{-1,\dotsc,n-1\}$. Note that $\cB_{2,n-2} = \cB_{2,n-1} = \varnothing$ since $\dim Z_{q,\zz} = n-3$, and $\cB_{2,-1} = \varnothing$ by Lemma \ref{lem:T_s-special}(ii) (recall that $\yy \neq \0 \neq \zz$ if $(\yy,\zz) \in \cB_2$). Furthermore, Lemma \ref{lem:T_s-special}(ii), property $(\R_1(q))$ and Proposition \ref{prop:trivial} imply that
\[
\# \cB_{2,0} \ll \left( \dfrac{B}{p}\right)^{n-1} \left( \dfrac{B}{\pi}\right)
\]
and
\[
\# \cB_{2,\sigma} \ll \left( \dfrac{B}{p}\right)^{n-1-\sigma} \left( \dfrac{B}{\pi}\right)^2
\]
for $\sigma \in \{1,\dotsc,n-3\}$. Inserting \eqref{eq:Delta(y,z)_2}, we get
\[
\sum_{(\yy,\zz) \in \cB_{2,\sigma}} \ll V_{\sigma,1} +V_{\sigma,2}
\]
for $\sigma \in \{0,\dotsc,n-3\}$, where
\begin{align*}
V_{\sigma,1} &= 
\begin{cases}
B^{2n} \pi^{-1}p^{-n+1}q^{-2} &\text{if } \sigma=0,\\
B^{2n+1-\sigma} \pi^{-2} p^{-n+1+\sigma}q^{-2} &\text{if } 1 \leq \sigma \leq n-3,
\end{cases}
\\
V_{\sigma,2} &=
\begin{cases}
B^{n+2} \pi^{-1}p^{-n+1}q^{(n-4)/2} &\text{if } \sigma=0,\\
B^{n+3} \pi^{-2} p^{-n+1+\sigma}q^{(n-4-\sigma)/2} &\text{if } 1 \leq \sigma \leq n-3.
\end{cases}
\end{align*}
As in the calculation of $\cE_{4,1}$, some of the $V_{\sigma,i}$ can be neglected when estimating $\cE_{4,2}$. We have
\begin{align*}
\sum_{(\yy,\zz) \in \cB_2} |\Delta(\yy,\zz)| &\ll V_{0,1} + V_{0,2} + V_{1,1} + V_{1,2} + V_{n-3,2},
\end{align*}
which implies that
\begin{equation}
\label{eq:E_4,2}
\begin{split}
\cE_{4,2} &\ll B^{3n/4} \pi^{(n-3)/4} p^{-1/4} q^{-1/2} + B^{(n+1)/2} \pi^{(n-3)/4} p^{-1/4} q^{(n-4)/8} \\
&\hphantom{\ll} +B^{3n/4} \pi^{(n-4)/4} q^{-1/2} + B^{(2n+3)/4} \pi^{(n-4)/4} q^{(n-5)/8} \\
&\hphantom{\ll} +B^{(2n+3)/4} \pi^{(n-4)/4} p^{(n-4)/4} q^{-1/8}. 
\end{split}
\end{equation}

Finally, we calculate $\cE_{4,3}$. For $(\yy,\0) \in \cB_3$, we estimate $|\Delta(\yy,\0)|$ using the case $r=1$ of Remark \ref{rem:Marmon}:
\begin{equation*}
\begin{split}
|\Delta(\yy,\zz)| &\leq N_{W_{\pi\yy,p\zz}}(f,B,q) + O(B^n q^{-3}) \\
&\ll B^n q^{-1} + B q^{(n-2)/2}.
\end{split}
\end{equation*}
Since $\# \cB_3 \ll (B/\pi)^n$, we get
\begin{equation}
\label{eq:E_4,3}
\cE_{4,3} \ll B^{3n/4}\pi^{-1/2} p^{(n-2)/4} q^{-1/4} + B^{(2n+1)/4} \pi^{-1/2} p^{(n-2)/4} q^{(n-2)/8}.
\end{equation}
Putting together the contributions from \eqref{eq:E_4,1}, \eqref{eq:E_4,2} and \eqref{eq:E_4,3}, we arrive at the estimate \eqref{eq:lem:E4}.
\end{proof}

\section{Proof of the main theorems}
\label{sec:proof}

Let $f$ be a polynomial in $\ZZ[x_1,\dotsc,x_n]$ of degree $d\geq4$ with leading form $F$, let  $Z=\Proj\ZZ[x_1,\dotsc,x_n]/(F)$, and suppose that $Z_\QQ$ is a non-singular subscheme of $\PP^{n-1}_\QQ$.
Note that Lemma \ref{lem:main} (\ref{ml1}) gives an asymptotic formula for $N_W(f,B,\pi pq)$. However, we shall only use it as an upper bound, and try to deduce a good upper bound for $N(f,B)$ by choosing $\pi$, $p$ and $q$ wisely in terms of $B$. It turns out that the following relations are desirable:
\begin{equation}
\label{eq:pipq-relations}
\begin{gathered}
\pi \asymp B^{(n^2-n-2)/(n^2+8n-4)},\\
p \asymp B^{(n^2-2n+8)/(n^2+8n-4)}, \\
q \asymp B^{2(n^2-n-2)/(n^2+8n-4)}, 
\end{gathered}
\end{equation}
since then the first, second and ninth terms in \eqref{eq:lem:E4} will be of the same order of magnitude as the main term in \eqref{eq:ml1}, and all other terms involved will be smaller. To be able to use Lemma \ref{lem:main} we need to have $q \gg B$, which is consistent with \eqref{eq:pipq-relations} as soon as $n \geq 10$. However, in case $n < 10$, the estimate in Theorem \ref{thm:main} follows already from \cite[Thm. 2]{Heath-Brown}. More importantly, the results of Lemma \ref{lem:main} are subject to a set of hypotheses \eqref{eq:R_i()} on $\pi, p, q$. We need to show that such $\pi, p, q$ exist in the specified intervals \eqref{eq:pipq-relations}. 

By Corollary \ref{cor:veryverygood}, however, the set of primes not fulfilling these criteria is finite. Thus,  Bertrand's postulate \cite[Theorem 418]{Hardy-Wright} assures that the intervals specified in \eqref{eq:pipq-relations}, with implied constants depending on $F$, contain primes satisfying \eqref{eq:R_i()}. We are thus allowed to insert \eqref{eq:pipq-relations} into Lemmata \ref{lem:main} and \ref{lem:E4}. Then we have, for the main term in \eqref{eq:ml1},
\[
 (\pi pq)^{-1} N_W(0,B,\pi pq) \ll B^n (\pi pq)^{-1} \ll_F B^{n-4+(37n-18)/(n^2+8n-4)}.
\]
The same holds for the 'main' auxiliary term - by Lemma \ref{lem:E4} we have
\[
\cE_4 \ll_F B^{n-4+(37n-18)/(n^2+8n-4)},
\]
where, as mentioned above, the first, second and ninth terms in \eqref{eq:lem:E4} dominate the expression. Thus, to finish the proof of Theorem \ref{thm:main} it remains to check that the remaining error terms occurring in Lemma \ref{lem:main} are small enough. We shall omit these calculations.

The key argument in the proof of Theorem \ref{thm:mainuniform} is the following lemma, which is a version of a result by Heath-Brown (see \cite[Thm. 4]{Heath-Brown02} and \cite[Lemma 5]{BHS}).

\begin{lemma}
\label{lem:Siegel}
Let $f \in \ZZ[x_1,\dotsc,x_n]$ be a polynomial of degree $d$. Let
\[
S(f,B) = \{\xx \in \ZZ^n; f(\xx)=0, |\xx| \leq B\}.
\]
Then one of the following holds:
\begin{enumerate}
\item
\label{enum:coefficients bounded}
There is a constant $\theta$, depending only on $n$ and $d$, such that $\Vert f \Vert \ll_{n,d} B^\theta$.
\item
\label{enum:auxpoly}
There exists a polynomial $g \in \ZZ[x_1,\dotsc,x_n]$ of degree $d$, not equal to $\lambda f$ for any $\lambda \in \QQ$, such that $g(\xx) = 0$ for every $\xx \in S(f,B)$.
\end{enumerate}
\end{lemma}

\begin{proof}
The result follows upon applying \cite[Lemma 5]{BHS} to the homogenization $F_0 \in \ZZ[x_0,\dotsc,x_n]$ of $f$. 
\end{proof}

Now we are able to prove Theorem \ref{thm:mainuniform}. Suppose that \eqref{enum:coefficients bounded} holds in Lemma \ref{lem:Siegel}. Then by Corollary \ref{cor:veryverygood} we have
\[
\prod_{p \in \cP(F)} p \ll_{n,d} B^{\theta\kappa}.
\]
Thus it is possible, using Bertrand's postulate, to find primes $\pi$, $p$ and $q$ satisfying \eqref{eq:R_i()} and \eqref{eq:pipq-relations}, with the implied constants in \eqref{eq:pipq-relations} depending only on $n$ and $d$. Now we proceed exactly as in the proof of Theorem \ref{thm:main}, except that all implied constants now depend only on $n$ and $d$. We get the following result.

\begin{thm}[Theorem \ref{thm:main}$'$]
\label{thm:main'}
Under the hypotheses of Theorem \ref{thm:main}, suppose furthermore that there is a constant $\theta \ll_{n,d} 1$ such that $\Vert f \Vert \ll_{n,d} B^\theta$. Then we have the estimate
\[
N(f,B)\ll_{n,d} B^{n-4+(37n-18)/(n^2+8n-4)}.
\]
\end{thm}

In case \eqref{enum:auxpoly} one easily sees that $N(f,B) \ll_{n,d} B^{n-2}$. To improve this to $B^{n-3+\epsi}$ requires some work, to which we devote Section \ref{sec:dimconj}. The estimate given by Theorem \ref{thm:dimgrowth} is enough to deduce Theorem \ref{thm:mainuniform}.

%
%

\section{Integral points on certain affine varieties}
\label{sec:dimconj}

The aim of this section is to prove the following result.

\begin{thm}
\label{thm:dimgrowth}
Let $n \geq 4$. Let $f \in \ZZ[x_1,\dotsc,x_n]$ be a polynomial of degree $d \geq 4$, whose leading form $F$ defines a non-singular hypersurface in $\PP^{n-1}_\QQ$. Let $g \in \ZZ[x_1,\dotsc,x_n]$ be another polynomial, not divisible by $f$. Then we have the estimate
\[
N(f,g,B) \ll_{n,d,\epsi} 
\begin{cases}
B^{n-3+1/12 + \epsi}, & 4 \leq n \leq 11,\\
B^{n-3+\epsi}, & n \geq 12.
\end{cases}
\]
\end{thm}

First we make some remarks on notation. Unless otherwise stated, we work over $\QQ$, that is, $\AA^n = \AA^n_\QQ$ and $\PP^n = \PP^n_\QQ$. We shall use the notation $V(\gamma_1,\dotsc,\gamma_r)$ for the closed subset of $\AA^n$ defined by $\gamma_1 = \cdots = \gamma_r = 0$, endowed with its reduced scheme structure. We denote by $H_0 \subset \PP^n$ the hyperplane defined by $x_0 = 0$. If $U \subseteq \AA^n$ is a locally closed subset, then we define
\[
U(\ZZ,B) = U(\QQ) \cap \ZZ^n \cap [-B,B]^n
\]
for any positive real number $B$, and $N(U,B) = \#U(\ZZ,B)$.

In proving Theorem \ref{thm:dimgrowth} we shall use results by Browning, Heath-Brown and Salberger \cite{BHS}. However, we shall need a slightly more general version of \cite[Thm. 2]{BHS}, which was shown to us by Salberger. 

\begin{thm}
\label{thm:BHS2}
Let $f \in \bar\QQ[x_1,x_2,x_3]$ be an irreducible polynomial of degree $d \geq 3$. Suppose that the leading form $F$ of $f$ has no irreducible factors of degree $1$ or $2$. Then
\[
N(f,B) \ll_{d,\epsi} 
\begin{cases}
B^{5/(3\sqrt{3}) + 1/4 + \epsi}, & d=3,\\
B^{3/(2\sqrt{d}) + 1/3 + \epsi}, & d=4 \text{ or } 5,\\
B^{1+\epsi}, & d \geq 6.
\end{cases}
\]
\end{thm}

We achieve this generalization by noting that the hypotheses in \cite[Lemma 9]{BHS} can be weakened. In the statement of that lemma, it suffices to assume that $X \cap H$ has no irreducible component of degree at most $e$. Indeed, that is enough to provoke the contradiction in the last line of the proof of the lemma. 


From this we immediately get the following strengthening of \cite[Prop. 1]{BHS}. 

\begin{prop}
\label{prop:BHSprop1}
Let $X \subset\PP^3_{\bar\QQ}$ be an integral surface of degree $d \geq 3$ such that every irreducible component of $X \cap H_0$ (with its reduced scheme structure) has degree at least $3$. Let $D=1$ or $2$, let $I_D$ be any finite set of integral curves $C \subset X$ of degree $D$, and let $\Sigma = \bigcup_{I_D} C$. Then we have the estimate
\begin{multline*}
\#\left\lbrace (x_1,x_2,x_3) \in \ZZ^3 \cap [-B,B]^3;\ (1:x_1,x_2,x_3) \in \Sigma(\QQ)\right\rbrace  \\
\ll_{d,\epsi} B^{\epsi} \max \{B^{2/d},B^{1/D},\#I_D\}
\end{multline*}
\end{prop}

In the proof of \cite[Thm. 2]{BHS}, one uses Heath-Brown's determinant method to prove that the integral points of height at most $B$ on the surface $S$ defined by $f=0$ are contained in $O_{d,\epsi}(B^{2/\sqrt{d}})$ curves  on $S$ of bounded degree. It is only in handling the contribution from lines and conics that the irreducibility of $F$ is used. For the rest of the proof, it suffices to assume that $f$ itself is irreducible. Thus, letting Proposition \ref{prop:BHSprop1} play the role of \cite[Prop. 1]{BHS} in the proof of \cite[Thm. 2]{BHS}, we immediately deduce Theorem \ref{thm:BHS2}.

\begin{proof}[Proof of Theorem \ref{thm:dimgrowth}]
Let $X = V(f) \subset \AA^n$ and let $Y$ be any integral component of $V(f,g) \subset \AA^n$. Note that $X$ is geometrically integral by assumption, so the dimension of $Y$ is $n-2$. We shall prove that 
\begin{equation}
\label{eq:dimconj}
N(Y,B) \ll_{n,d,\epsi}
\begin{cases}
B^{n-3+1/12 + \epsi}, & 4 \leq n \leq 11,\\
B^{n-3+\epsi}, & n \geq 12.
\end{cases}  
\end{equation}

Let $\bar X, \bar Y \subset \PP^n$ be the respective projective closures. Let $X_0 = \bar X \cap H_0$ and $Y_0 = \bar Y \cap H_0$. Our hypothesis implies that $X_0$ is non-singular.  Then, as observed in \cite[Lemma 6.2]{Salberger05}, any closed subscheme of $(X_0)_{\bar \QQ}$ of pure codimension one is the intersection of $(X_0)_{\bar \QQ}$ with a hypersurface $G \subset \PP^{n-1}$. This is a consequence of the Noether-Lefschetz theorem (use \cite[Cor. 3.3, p. 180]{Hartshorne70} and the following exercises). Thus, every integral component of $(Y_0)_{\bar \QQ}$ has degree divisible by $d$, and in particular $d \mid \deg Y$.

Let $d' = \deg Y$. We can assume that $Y$ is geometrically integral. Indeed, in case $Y$ is integral, but not geometrically integral, one can argue as in the proof of \cite[Thm. 2.1]{Salberger05} to conclude that all rational points on $Y$ lie on a proper subvariety, obtained as the intersection of all the irreducible components of $Y_{\bar \QQ}$. Thus we can use a trivial estimate (e.g. Proposition \ref{prop:trivial} for suitably chosen $q$) to conclude that
\begin{equation}
\label{eq:trivial}
N(Y,B) \ll_{n,d} B^{n-3}.
\end{equation}

We have the following result, which we shall prove in Section \ref{subsec:projection}.  

\begin{prop}
\label{prop:projection}
Let $X \subset \AA^n_{\bar \QQ}$ be an integral closed subvariety of dimension $m \leq n-2$ and degree $d$. Let $\bar X \subset \PP^n_{\bar \QQ}$ be its projective closure and $X_0 = \bar X \cap H_0$. Then there exists an $((m+1) \times n)$-matrix $A$ and an $(m+1)$-vector $\bb$, with integer entries of size $O_{n,d}(1)$, such that the morphism
$\pi: X \to \AA^{m+1}_{\bar \QQ}$ given by $\xx \mapsto A \xx + \bb$ is birational onto its image. In particular, its fibres consist of at most $d$ points, and $\pi(X)$ is an integral closed hypersurface of degree $d$.

Moreover, it induces a morphism $\bar\pi: \bar X \to \PP^{m+1}_{\bar\QQ}$ with the following property. If
\[
X_0 = X_{0,1} \cup \dotsb \cup X_{0,k}
\]
is the decomposition of $X_0$ into irreducible components and
$\deg X_{0,i} = d_i$, then $\bar\pi:X_{0,i} \to \PP^{m+1}_{\bar\QQ}$ is birational onto its image for each $i$. In particular 
\[
\bar\pi(X) \cap H_0 = \bar\pi(X_0) = \bar\pi(X_{0,1}) \cup \dotsb \cup \bar\pi(X_{0,k}),
\]
where each $\bar\pi(X_{0,i})$ is integral of dimension $m-1$ and degree $d_i$.  
\end{prop} 

Thus we find a geometrically integral hypersurface $W =\pi(Y) \subset \AA^{n-1}$ of degree $d'$ such that $N(Y,B) \leq d \cdot N(W,cB)$ for some constant $c \ll_{n,d} 1$, and such that if $\bar W \subset \PP^{n-1}$ is the projective closure and $W_0 = \bar W \cap H_0$ (taken with its reduced scheme structure), then $(W_0)_{\bar \QQ}$ has no irreducible component of degree less than $d$.

It is a standard fact \cite[Lemma 7]{BHS} that we can find a hyperplane $H \subset \PP^{n-1}$, defined by a linear form with integer coefficients of size $O_{n,d}(1)$, such that the intersection of $H$ with $W$ or any of the irreducible components of $(W_0)_{\bar \QQ}$ is again irreducible. Indeed, the set $\cE \subset \check{\PP}^{n-1}$ of hyperplanes $H$ such that this fails is a proper closed subscheme of degree $O_{n,d}(1)$. After a suitable change of variables (sending $H_0$ to itself), we can assume that $H$ is given by $x_{n-1}=0$. Letting $H_a \subset \AA^{n-1}$, for any $a \in \ZZ$ be the hyperplane given by $x_{n-1} = a$, and putting $W_a = W \cap H_a$, we have
\[
N(W,B) = \sum_{a=-c'B}^{c'B} N(W_a,B).
\]
For all but $O_{n,d}(1)$ values of $a$, $\bar W_a$ is geometrically irreducible, and $\bar W_a \cap H_0$ has no irreducible components over $\bar \QQ$ of degree less than $d$. Indeed, let $\cH$ be the linear pencil of hyperplanes $\lambda H + \mu H_0$ parameterized by $(\lambda:\mu)\in \PP^1$. Since $\cH \nsubseteq \cE$, we have $\dim (\cH \cap \cE) = 0$. The exceptional values of $a$ yield an acceptable contribution to \eqref{eq:dimconj} by a trivial estimate for $N(W_a,B)$. 

Applying this process inductively, much as in \cite[\S 4]{BHS}, we find a collection of $O_{n,d}(B^{n-4})$ geometrically irreducible surfaces $S \subset \AA^3$ of degree $d'$ such that the curve $S_0 = \bar S \cap H_0$ has no components over $\bar \QQ$ of degree less than $d$, and such that the estimate
\[
N(W,B) \leq \sum_{S} N(S,c''B) + O_{n,d}(B^{n-3}).
\]
holds for some constant $c'' \ll_{n,d} 1$.

There are now two cases to consider. 

\textbf{Case 1:}
$d' \geq 2d$. Then Theorem \ref{thm:BHS2} yields the estimate
\[
N(S,B) \ll_{d,\epsi} B^{1+\epsi},
\]
which suffices to establish the desired bound for $N(Y,B)$.

\textbf{Case 2:} $d' = d$. Then the estimate given by Theorem \ref{thm:BHS2} is
\[
N(Y,B) \ll_{n,d,\epsi} B^{n-3+1/12+\epsi}.
\]

For large $n$, we shall derive a better estimate by applying our main result inductively. 

In the present case we necessarily have $Y_0 = X_0 \cap \Gamma_0$ for some hyperplane $\Gamma_0 \subset \PP^{n-1}$. But then we must also have $\bar Y = \bar X \cap \Gamma$ for some hyperplane $\Gamma \subset \PP^{n}$. 

(Indeed, let $\mathcal G$ be the family of $\Gamma \in \GG(n-1,n)$ such that $\Gamma_0 \subset \Gamma$. Then $Y_0 \subseteq \bar Y \cap \Gamma$ for every $\Gamma \in \mathcal G$, and the inclusion has to be strict for some $\Gamma$. If $\Gamma$ would intersect $\bar Y$ properly, then we would have $\bar Y \cap \Gamma = Y_0 \cup Z$ for some closed subscheme $Z \subset \bar Y$ of codimension one. But this would contradict the fact that $\deg (\bar Y \cap \Gamma) = d = \deg Y_0$. Thus we conclude that $\bar Y \subseteq \Gamma$.)

Thus, in this case we have $Y = X \cap L$ for some hyperplane $L \subset \AA^n$. Now $\Lambda := L \cap \ZZ^n$ is a lattice of dimension $r \leq n-1$. Thus, by \cite[Lemma 1]{Heath-Brown02}, it has a basis $\bb_1,\dotsc,\bb_r$ such that for every $\xx = \sum \lambda_j \bb_j \in \Lambda \cap [-B,B]^n$ we have $\lambda_j \ll |\xx|/|\bb_j|\ll B$. Thus we get a bijection between $Y(\ZZ,B)$ and $Y'(\ZZ,cB)$, where $c \ll_{n,d} 1$, $h(\yy) = f(\sum y_j \bb_j)$ and $Y' = V(h) \subset \AA^{r}$.

By Lemma \ref{lem:Siegel} we can assume that $\Vert h \Vert \ll_{n,d} B^\theta$, since in case \eqref{enum:auxpoly} we can use a trivial estimate as in \eqref{eq:trivial} to conclude that
\begin{equation}
N(Y,B) \leq N(h,cB) \ll_{n,d} B^{r-2} \leq B^{n-3}.
\end{equation}

Since $X_0$ is non-singular, it is well known \cite[Appendix, Thm. 2]{Hooley} that $Y_0$ can have at most isolated singularities. Thus, the same holds for the closed subscheme $Y'_0 \subset \PP^{r}$ defined by the leading form of $h$. Then we can find a hyperplane $\Pi \subset \PP^r$, defined by a linear form with integer coefficients of size $O_{n,d}(1)$, such that $Y'_0 \cap \Pi$ is non-singular. For a proof of this 'effective' version of Bertini's theorem, see \cite[Lemma 2.8]{Marmon}. After a suitable linear transformation we may assume that $\Pi$ is given by $y_r = 0$. For any $a \in \ZZ$, let $h_a \in \ZZ[y_1,\dotsc,y_r]$ be given by $h_a(y_1,\dotsc,y_{r-1}) = h(y_1,\dotsc,y_{r-1},a)$. All the polynomials $h_a$ then have the same non-singular leading form. Thus we can apply the estimate of Theorem \ref{thm:main'} to get
\[
N(h_a,B) \ll_{n,d} B^{n-6 + (37(n-2) -18)/((n-2)^2 + 8(n-2)-4)}.
\]
This yields an estimate
\begin{align*}
N(Y,B) &\leq N(h,cB) \ll_{n,d} \sum_{a=-cB}^{cB} N(h_a,cB) \\
&\ll_{n,d} B^{n-5 + (37n-92)/(n^2 + 4n - 16)}.
\end{align*}
As soon as $n \geq 12$, we get $N(Y,B) \ll_{n,d} B^{n-3}$.
\end{proof}


\subsection{Birational projections of bounded height}
\label{subsec:projection}

Let $Z \subset \PP^n_{\bar \QQ}$ be an integral closed subvariety of dimension $m \leq n-2$. Throughout this section, we work over $\bar \QQ$, but henceforth we shall omit this subscript. Let $\Lambda$ be an $(n-m-2)$-plane and $\Gamma$ an $(m+1)$-plane such that $\Lambda \cap \Gamma = \varnothing$. We recall the construction of the projection 
\[
\pi_{\Lambda,\Gamma}:\PP^n \setminus \Lambda \to \Gamma
\]
from $\Lambda$ to $\Gamma$ (see \cite[Lecture 3]{Harris}). Identifying $\Gamma$ with $\PP^{m+1}$, we write $\pi_\Lambda:\PP^n \dashrightarrow \PP^{m+1}$. It is known that for a generic $\Lambda \in \GG(n-m-2,n)$, the projection $\pi_{\Lambda}\vert_Z:Z \to \PP^{m+1}$ is birational onto its image. In particular $\pi(Z)$ is integral and $\deg \pi_\Lambda(Z) = \deg Z$. In \cite[\S 3]{BHS} it is shown that one can also find such a projection where $\Lambda$ is defined over $\QQ$ and of bounded height. We shall need an affine version of that statement. 

Let us recall the notation of \cite[\S 3]{BHS}. Let $Z \subset \PP^n$ be a closed subvariety of dimension $m$ and degree $d$. For any $\Lambda \in \GG(n-m-2,n)$, define
\begin{gather*}
S_{\Lambda,Z} = \{ M \in \GG(n-m-1,n); \Lambda \subset M, \#(M \cap Z) \geq 2 \},\\
Y_Z = \{\Lambda \in \GG(n-m-2,n); \dim S_{\Lambda,Z} \geq m\}, \\
Y'_Z= \{\Lambda \in \GG(n-m-2,n); \Lambda \cap Z \neq \varnothing \}
\end{gather*}
Then, in case $Z$ is integral, it is shown \cite[Lemma 6]{BHS} that $Y_Z$ and $Y'_Z$ are proper closed subvarieties of $\GG(n-m-2,n)$ of degree $O_{n,d}(1)$. Furthermore, as soon as $\Lambda \notin Y_Z \cup Y'_Z$, the morphism $\pi_{\Lambda}:Z \to \PP^{m+1}$ is birational onto its image, and its fibres consist of at most $d$ points. 

Let now $X \subset \AA^n$ be an integral closed subvariety of dimension $m$ and degree $d$, and let $Z \subset \PP^n$ be its projective closure. As soon as $\Lambda \subset H_0$, the projection $\pi_\Lambda$ maps $\AA^n = \PP^n \setminus H_0$ into $\AA^{m+1}$. Thus, let
\[
G_0 = \{ \Lambda \in \GG(n-m-2,n); \Lambda \subset H_0\} \cong \GG(n-m-2,n-1).
\]
It is easy to see that
\begin{equation}
\label{eq:exceptionalG0}
(Y_Z \cup Y'_Z)\cap G_0 \subseteq Y_{Z_0} \cup Y'_{Z_0},
\end{equation}
where $Z_0 = Z \cap H_0$. Applying the arguments of \cite{BHS}, we deduce that $Y_{Z_0} \cup Y'_{Z_0}$ is a proper closed subvariety of $\GG(n-m-2,n-1)$ of degree $O_{n,d}(1)$. However, since $Z_0$ is not necessarily integral, this requires the following generalization of \cite[Lemma 6]{BHS}, the proof of which is straightforward:

\begin{lemma}
\label{lem:BHSLemma6'}
Suppose that the closed subvariety $Z\subseteq \PP^n$ is equidimensional of dimension $m$ and degree $d$. Then $Y_Z$ is a proper closed subvariety of $\GG(n-m-2,n)$ of degree $O_{d,n}(1)$.
\end{lemma}

By \cite[Lemma 3]{BHS} we can therefore find an $(n-m-2)$-plane $\Lambda \notin Y_{Z_0} \cup Y'_{Z_0}$ that is defined over $\QQ$ and of bounded height. The projection $\pi_\Lambda:Z \to \PP^{m+1}$ is then birational onto its image. Moreover,  $\pi_\Lambda:Z_{0,i} \to \PP^{m+1}$ is birational onto its image for each irreducible component $Z_{0,i}$ of $Z_0$.

Finally, choosing $\Gamma$ as explicitly described in \cite{BHS}, it is evident that $\pi_\Lambda$ maps integral points of height at most $B$ in $\AA^n$ to integral points of height $O_{n,d}(B)$ in $\AA^{m+1}$. This finishes our proof of Proposition \ref{prop:projection}.

\bibliographystyle{plain}
\bibliography{ratpoints}
\end{document}